\documentclass{amsart}  
\usepackage{amssymb,amstext,amsmath,amscd,amsthm,amsfonts,enumerate,graphicx,latexsym,stmaryrd,multicol}   
\usepackage{scalerel}
\newcommand\reallywidehat[1]{\arraycolsep=0pt\relax%
\begin{array}{c}
\stretchto{
  \scaleto{
    \scalerel*[\widthof{\ensuremath{#1}}]{\kern-.5pt\bigwedge\kern-.5pt}
    {\rule[-\textheight/2]{1ex}{\textheight}} 
  }{\textheight} %
}{0.5ex}\\           
#1\\                 
\rule{-1ex}{0ex}
\end{array}
}
\usepackage{tikz-cd}
\usepackage{color}
\usepackage{hyperref} 
\usepackage{comment} 

\usepackage{amsmath, mathtools}
\hypersetup{colorlinks=true} 

\usepackage[all]{xy}
\newcommand{\com}[1]{{\color{magenta} #1}}
\newcommand{\old}[1]{{\color{red} #1}}
\newcommand{\new}[1]{{\color{blue} #1}}
\newtheorem{thm}{Theorem}[section]
\newtheorem{lem}[thm]{Lemma}
\newtheorem{prop}[thm]{Proposition}
\newtheorem{cor}[thm]{Corollary}
\newtheorem{chunk}[thm]{\hspace*{-1.065ex}\bf}

\theoremstyle{definition}
\newtheorem{dfn}[thm]{Definition}
\newtheorem{ques}[thm]{Question}
\newtheorem{rem}[thm]{Remark}

\theoremstyle{remark}

\numberwithin{equation}{thm}

\def \c {\mathfrak c} 
\def\depth{\operatorname{depth}} 
\def \n {\mathfrak n}

\def\tr{\operatorname{tr}}
\def\Ext{\operatorname{Ext}}

\def\hom{\operatorname{Hom}} 
\def\add{\operatorname{add}} 
\def\m{\mathfrak{m}}
\def \ind {\operatorname{ind}}

\def\ann{\operatorname{ann}}    
\def \C {\mathcal C} 

\def\mod{\operatorname{mod}} 
\def\X{\mathcal{X}}
\def\Y{\mathcal{Y}} 
\def\supp{\operatorname{Supp}}
\def\p{\mathfrak{p}} 
\def\Hom{\operatorname{Hom}} 
\def \Ul {\operatorname{Ul}}

\def \cm {\operatorname{CM}} 
\def\syz{\Omega} 
\def\spec{\operatorname{Spec}}
\def\q{\mathfrak{q}}
\def \Ass {\operatorname{Ass}} 
\def \Min {\operatorname{Min}}
\def \Mod {\operatorname{Mod}} 
\def \End {\operatorname{End}} 
\def \us {\mathcal{US}} 
\begin{document}  
\title{Ulrich split rings}  

\author{Hailong Dao}
\address{H. Dao, Department of Mathematics, University of Kansas, KS 66045-7523, Lawrence, USA} 
\email{ hdao@ku.edu}  

\author{Souvik Dey}
\address{S. Dey, Department of Algebra, Charles University, Faculty of Mathematics and Physics, Sokolovska
83, 186 75, Praha, Czech Republic}
\email{ souvik.dey@matfyz.cuni.cz} 

\author{Monalisa Dutta}
\address{M. Dutta, Department of Mathematics, University of Kansas, KS 66045-7523, Lawrence, USA} \email{ m819d903@ku.edu}  


\thanks{Souvik Dey was partly supported by the Charles University Research Center program No.UNCE/SCI/022 and a grant GA \v{C}R 23-05148S from the Czech Science Foundation.}   
 
\begin{abstract} 
A local Cohen--Macaulay ring is called Ulrich split if any short exact sequence of Ulrich modules split. In this paper we initiate the study of Ulrich split rings. We prove several necessary or sufficient criteria for this property, linking it to syzygies of the residue field and cohomology annihilator. We characterize Ulrich split rings of small dimensions. Over complex numbers, $2$-dimensional Ulrich split rings, which are normal and have minimal multiplicity, are precisely cyclic quotient singularities with at most two indecomposable Ulrich modules up to isomorphism.  We give several ways to construct Ulrich split rings, and give a range of applications, from test ideal of the family of maximal Cohen--Macaulay modules, to detecting projective/injective modules via vanishing of $\Ext$.      
\end{abstract}       
\maketitle 

\section{Introduction}  

Let $(R,\m,k)$ be a local  Cohen--Macaulay ring, let $\mod R$ denote the category of all finitely generated $R$-modules. Recall that, a module over $R$ is called Ulrich if it is maximal Cohen--Macaulay and its Hilbert-Samuel multiplicity $e(M)$ and minimal number of generators $\mu(M)$ are equal. Despite the simple definition, Ulrich modules (as well as Ulrich sheaves and vector bundles) have become rather well-studied topics in commutative algebra and algebraic geometry. The literature is rather vast and fast-changing, so let us just mention a sample of papers from commutative algebra that are close in spirit to this present work \cite{ul, umm, cyc}, some modern surveys on the geometric point of view \cite{surveyB, coskun} and the references therein. There have also been quite a few very recent extensions of the concept - Ulrich modules with respect to an ideal, or lim Ulrich sequences - with rather wide applications, see \cite{dms, gotww1, gotww2, imk, MA}.     

Let $\Ul(R)$ denote the subcategory of $\mod{R}$ that consists of all Ulrich $R$-modules. In this paper, we study local rings $R$, whose $\Ul(R)$ is simplest from the point of view of "relative homological algebra". It can be observed that Ulrich modules, together with the short exact sequences of them, form an exact category, see \cite[Corollary 4.7]{ddd}. We say that a ring $R$ is Ulrich split, abbreviated by $\us$, if all short exact sequences of Ulrich $R$-modules split. In other words, the exact structure of the exact category $\Ul(R)$ is trivial (here we abuse notations slightly and use $\Ul(R)$ also for the exact category).  

Our motivation comes partly from the recent surge of interest in Ulrich modules, and from the long-observed phenomenon that there are tight and still mysterious connections between homological/categorical and algebro-geometric properties of local singularities. For instance, from the point of view of this paper, the famous Auslander-Buchsbaum-Serre Theorem implies that, a local Cohen--Macaulay ring $R$ is regular if and only if all short exact sequences of maximal Cohen--Macaulay $R$-modules split. Thus, the analogous property for Ulrich modules seems to deserve some attention.  

If $(R,\m,k)$ is Artinian, then an $R$-module is Ulrich if and only if it is a direct sum of copies of $k$. Thus, all Artinian local rings are $\us$. The property becomes much more subtle in higher dimensions, and there are surprising connections to other homological or geometric properties of $R$. As is usually the case, the story is far from complete, and we hope our work will lead to further inquests.  

Our first main result Theorem \ref{thm1} establishes several useful necessary and sufficient conditions for being $\us$. This is contained in Proposition \ref{syzind}, Proposition \ref{ulind}, Theorem \ref{uladdsplit}, Theorem \ref{minuss}. In the following,  for a subcategory $\X$ of $\mod R$, $\add_R \X$ will denote the smallest subcategory of $\mod R$ containing $\X$, closed under direct summands of finite direct sums. It is clear that, $M\in \add_R \X$ if and only if $M\oplus N \cong \oplus_{i=1}^n X_i$, for some $N \in \mod R$, and $X_i \in \X$. Moreover, $\syz^i_R(-)$ denotes the $i$-th syzygy in a minimal $R$-free resolution.  

\begin{thm}\label{thm1} Let $(R,\m,k)$ be a local Cohen--Macaulay ring of dimension $d$. Consider the following conditions:

\begin{enumerate}[\rm(1)]
    \item  $\Ul(R)\subseteq \add_R\{R, \syz^i_R k: i\ge 0\}$.   
    
    \item $\Ul(R)\subseteq \add_R\{\syz^i_R k: i\ge d\}$.   
    
    
    \item $R$ is $\us$.  
    
    \item $\m\Ext^1_R(\Ul(R), \Ul(R))=0$. 
\end{enumerate} 

Then it holds that $(1)\iff (2) \implies (3) \implies (4)$. If $R$, moreover, has minimal multiplicity, then the above conditions are all equivalent and are in turn equivalent to $(5)$ $\Ul(R)=\add_R(\syz^d_R k)$. If, along with minimal multiplicity, $R$ moreover admits a canonical module, then the above conditions are all equivalent to $(6)$ $\m\Ext^{>0}_R(\cm(R),\mod R)=0$. 
\end{thm}

Proposition \ref{syzind} and Proposition \ref{ulind} together proves Theorem \ref{thm1}($(1)\implies (2)$).   
Since the $\us$ property of $R$ is equivalent to saying $\Ext^1_{\Ul(R)}(M,N)=0$ for all $M,N\in \Ul(R)$ (in the sense of \cite[4.9, Corollary 4.7]{ddd}), we note that $(3)\implies (4)$ in Theorem \ref{thm1} is nothing but \cite[Lemma 5.2.2]{ddd} which says that $\m \Ext^1_R(M,N)\subseteq \Ext^1_{\Ul(R)}(M,N)$ for all $M,N\in \Ul(R)$. 

We next turn our attention to rings of small dimensions. For dimension $1$ local Cohen--Macaulay rings, we can give a complete characterization of Ulrich split rings, under mild conditions. They are precisely the rings, for whom the blow-up ring with respect to the maximal ideal is regular. As will be seen below, the result holds much more generally for $I$-Ulrich modules in the sense of \cite[Definition 4.1]{dms}, i.e., $I$-Ulrich modules are maximal Cohen--Macaulay modules $M$ satisfying $e(I,M)=\ell(M/IM)$. Here, we note that Ulrich modules are nothing but $\m$-Ulrich modules. In the result below, for an ideal $I$ of $R$, $R \bowtie I$ will denote the fibre product of $R$ with itself along $R/I$. The equivalence (1) $\iff$ (5) of Theorem \ref{1.2} is \cite[Proposition 5.2.6]{ddd}, and the equivalence of the rest is contained in Theorem \ref{dim1uss}, Corollary \ref{complete}, Corollary \ref{dimsyzann}, Proposition \ref{fibbl}.  In the following, for an $R$-module $M$, $\tr_R(M)$ will stand for $\sum_{f\in \Hom_R(M,R)}f(M)$. When the ring in question is clear we will drop the suffix $R$ and may simply write $\tr_R(M)$. Moreover, $\syz\cm^{\times}(R)$  denotes the collection of modules which is a syzygy in a minimal free resolution of some MCM $R$-module (see \cite[Notation 3.3]{umm}, \cite[Lemma 2.2]{del}).  

\begin{thm}\label{1.2} Let $(R,\m,k)$ be a local Cohen--Macaulay ring of dimension $1$. 
Let $I\subseteq \m$ be an $\m$-primary ideal, and $\Ul_I(R)$ denote the subcategory of $I$-Ulrich modules. Let $B(I):=\cup_{n\geq 1}(I^n:_{Q(R)}I^n)$ denote the blow-up algebra of $I$, let $\c$ be the conductor of the integral closure of $R$. Then the following are equivalent:  

\begin{enumerate}[\rm(1)]   
    \item Every short exact sequence of modules in $\Ul_I(R)$ splits.

    \item Every short exact sequence of modules in $\Ul_{I\times I}(R \bowtie I)$ splits.  
    
    \item The blow-up $B(I)$ is a regular ring. 
    
    \item For all $n\gg 0$, $\tr_R(I^n)=\c$ and $I^n$ is a reflexive $R$-module.  
\vspace{1mm}

When $I=\m$, then the above equivalent conditions are also equivalent to 
\vspace{1mm}

\item $\m \Ext^1_R(\Ul(R),\Ul(R))=0$.  

\item The completion $\widehat R$ is $\us$.   

\vspace{1mm} 

When, moreover, $R$ is singular and has minimal multiplicity, then the above conditions are also equivalent to

\vspace{1mm}  

\item $\m = \c$.

\item $R$ is analytically unramified and $\syz \cm^{\times}(R)=\add_R(\m)$.  

\item $R$ is analytically unramified and $\m \Ext^1_R(\c, \syz_R \c)=0$.  
\end{enumerate}  
\end{thm}

In dimension two, under the assumptions of minimal multiplicity and complex residue field, we can characterize $\us$ rings as cyclic quotient singularities with at most two indecomposable Ulrich modules up to isomorphism.      

\begin{thm}[Theorem \ref{surface}]\label{surface2} Let $(R,\m,k)$ be a complete local Cohen--Macaulay ring of minimal multiplicity, and have dimension $2$. Consider the following two statements:
 
 \begin{enumerate}[\rm(1)]  
  \item $R$ is a cyclic quotient singularity and $\Ul(R)$ has at most two indecomposable objects.
     \item $R$ is $\us$. 
 \end{enumerate}
 
Then $(1)\implies (2)$. If moreover $k=\mathbb C$, then $(2)\implies (1)$.
\end{thm}   

As a consequence, we are also able to characterize the cyclic surface quotient singularities for which the test ideal $\tau_{\cm(R)}(R)$ in the sense of \cite[Definition 4.1, Corollary 4.4]{pr} is the maximal ideal. For basic notations involving cyclic surface quotient singularities as used in Theorem \ref{thm4}, we refer the reader to \cite[Section 2 Preliminaries
]{cyc}.    

\begin{thm}[Lemma \ref{trcl}, Proposition  \ref{trquot}]\label{thm4} Let $R=\mathbb C[[x,y]]^G$ be a cyclic quotient surface singularity with maximal ideal $\m$. Then the following are equivalent.
 
 \begin{enumerate}[\rm(1)]

\item $G=\dfrac 1 n(1,a)$, where $a$ divides $n+1$. 
 


     \item $\tau_{\cm(R)}(R)=\m$.
     
     \item $\tr_R(M)=\m$ for every Ulrich $R$-module $M$.
 \end{enumerate}  
 \end{thm}  

 
We briefly describe the organization of the paper. Section 2 takes care of preliminary results. While some of them are familiar, the latter ones dealing with categorical matters, from \ref{syzind} onwards, are more technical. For instance, we need to use results from our previous paper on subfunctors of $\Ext$ (\cite{ddd}) to establish splitting from a certain condition on the Grothendieck group,  see \ref{Ggroup} and \ref{remark2.15}. The short section 3 establishes several results on the relationship between Ulrich modules and syzygies of the residue field. Section 4 explores rings of minimal multiplicity more deeply, and this is where the main results regarding these rings and the $\us$ property are completed, in particular Theorems \ref{thm1} and \ref{surface2}. Section 5 focuses on dimension one, where very comprehensive statements are available, collected in Theorem \ref{1.2} above. Section 6 gives some applications for detecting projective and injective modules via the vanishing of $\Ext$, where the ring being $\us$ often means one can  require less vanishing than usual.   

\section{Preliminaries} 

 In this section, we record many preliminary results in a very general set-up, which will be used throughout the various later sections in this paper.  Throughout this paper, all rings are commutative and Noetherian. Given a ring $R$, $\Mod R$ (resp. $\mod R$) denotes the category of all $R$-modules (resp. finitely generated $R$-modules). All subcategories are assumed to be strict and full (i.e., subcategories are only determined by their objects, and moreover, are closed under isomorphism).
  $Q(R)$ will denote the total ring of fractions of $R$ and $\overline R$ will stand for the integral closure of $R$ in $Q(R)$. When $R$ is local, we denote the first syzygy in a minimal free resolution of a finitely generated $R$-module $M$ by $\syz_R M$. By a regular ideal $I$ of a ring $R$, we will mean an ideal containing a non-zero-divisor of $R$.  
 
 For $R$-submodules $M,N$ of $Q(R)$, we denote $(M:N):=\{r\in Q(R): rN\subseteq M\}$. An extension of rings $R\subseteq S$ will be called birational if $S\subseteq Q(R)$, i.e., if $Q(S)=Q(R)$.

\begin{chunk}\label{conductor} Given a birational extension $R\subseteq S$, we define the conductor ideal of $S$ in $R$, denoted by $c_R(S)$, to be $(R:S)=\{r\in Q(R):rS\subseteq R\}$. Since $1\in S$, so $c_R(S)=(R:_R S)$. Notice that, $c_R(S)$ is an ideal of $R$. Moreover, since $S\cdot c_R(S) \cdot S\subseteq S\cdot c_R(S)\subseteq R$, so $S\cdot c_R(S)\subseteq (R:S)=c_R(S)$. Hence $c_R(S)$ is always an ideal of $S$ as well. We write $\mathfrak c$ to denote $c_R(\overline R)$. Since $c_R(S)=(R:S)=(R:S)S=\tr_R(S)$ (\cite[Proposition 2.4(2)]{trace}), hence $c_R(S)$ is always a trace ideal of $R$.  

Now, $S$ is module finite over $R$ if and only if the conductor ideal $ c_R(S)$ contains a non-zero-divisor of $R$ (i.e., $ c_R(S)\in \mathcal F_R$ in the notation of \cite{goto}). So, if $S$ is a finite birational extension of $R$, then $(R:c_R(S))=(c_R(S):c_R(S))$ by \cite[Proposition 2.4(3)]{trace}.   
\end{chunk}    

In the following, $M^*$ denotes $\Hom_R(M,R)$ for an $R$-module $M$.

\begin{lem}\label{normal} Let $R$ be an analytically unramified local Cohen--Macaulay ring (hence $\overline R$ is module finite over $R$). Then $\overline R=(\mathfrak c : \mathfrak c)=(R:\mathfrak c)\cong \hom_R(\mathfrak c, R)\cong \operatorname{End}_R(\mathfrak c)$ is a reflexive $R$-module.
\end{lem}    

\begin{proof} Since $R$ is reduced, so $R$ is generically Gorenstein. Since for any $M\in \mod R$, there is an exact sequence of $R$-modules $0\to M^*\to F_1 \to F_0$, where $F_0, F_1$ are free, so $M^*$ is reflexive for any finite $R$-module $M$ by \cite[Theorem 2.3]{matsui}. Hence the reflexivity of $\overline{R}$ follows if we can show the required isomorphisms. Let $x\in \overline{R}$. Then $x(R:\overline{R})\subseteq R$, so $x\in R:(R:\overline{R})=R:\mathfrak{c}$. Thus $\overline{R}\subseteq R:\mathfrak{c}$. Now $\overline R$ is module finite over $R$, so $\mathfrak c$ contains a non-zero-divisor of $R$. Hence \cite[Proposition 2.4(1)]{trace} implies $R:\mathfrak{c}$ is isomorphic to $\Hom_R(\mathfrak{c},R)$ as $R$-modules. Since $\mathfrak c$ is a trace ideal (by  \ref{conductor}) and $\c$ contains a  non-zero-divisor of $R$, so by \cite[Proposition 2.4(3)]{trace} we get, $\mathfrak c: \mathfrak c=R:\mathfrak c$. Thus $R:\mathfrak{c}=\mathfrak{c}:\mathfrak{c}$ is a finite birational extension of $R$ (since $\mathfrak c$ is an ideal of $R$, hence $\mathfrak c$ is a finite $R$-module), so $R:\mathfrak{c}\subseteq\overline{R}$. Hence $\overline{R}=R:\mathfrak c=\c : \c$. Thus by \cite[Proposition 2.4(1)]{trace} we have, $\text{End}_R(\mathfrak c) \cong \overline R$.   
\end{proof}    

\begin{lem}\label{algmor} Let $f:R \to S$ be an $R$-algebra. Then for any $M,N\in \Mod(S)$, it holds that $M,N\in \Mod(R)$ (via $f$), and moreover, $\Hom_R(M,N)=\Hom_S(M,N)$ holds in the following cases:
\begin{enumerate}[\rm(1)]
\item $f$ is surjective. 

\item $f$ is injective, $S\subseteq Q(R)$, and $N$ is a torsion-free $S$-module.   
\end{enumerate}
\end{lem}     

\begin{proof} If $M$ is an $S$-module, then the $R$-module structure on $M$ is given by $r\cdot x:=f(r)x$ for all $r\in R$, $x\in M$. Now to show $\Hom_R(M,N)=\Hom_S(M,N)$, first notice that $\Hom_S(M,N)\subseteq \Hom_R(M,N)$ is obvious. So, we only need to verify the other inclusion. 

(1) Let $\psi \in \Hom_R(M,N)$. Now to show $\psi$ is $S$-linear, let $s\in S$ and $x\in M$. As $f$ is surjective, so $s=f(r)$ for some $r\in R$. So, $s\psi(x)=f(r)\psi(x)=r\cdot \psi(x)=\psi(r\cdot x)=\psi(f(r)x)=\psi(sx)$.  

(2) This is \cite[Lemma 5.2.9]{ddd}.    
\end{proof}   

As a quick consequence, we get the following:

\begin{cor} \label{homalgebra} Let $f:R \to S$ be an $R$-algebra such that $f$  factors as $R \to A \to S$, where $R\to A$ is a surjective ring homomorphism and $A\to S$ is a birational embedding. If $M,N$ are $S$-modules and $N$ is torsion-free over $S$, then  $\Hom_R(M,N)=\Hom_S(M,N)$. So, in particular, a torsion-free $S$-module $M$ is indecomposable as an $S$-module if and only if it is indecomposable as an $R$-module.   
\end{cor}  

\begin{rem} Some assumptions on the map $f: R\to S$ in the Corollary \ref{homalgebra} are absolutely necessary. To see how things can go wrong even in very simple situations, consider the map from $n$-th Veronese $R=k[[x,y]]^{(n)}\to k[[x,y]]=S$, where $k[[x,y]]^{(n)}:=k[[x^iy^{n-i}: 0\leq i\leq n]]$. Then $S$ is obviously indecomposable as an $S$-module but is decomposable as an $R$-module as soon as $n>1$.   
\end{rem}  

\begin{lem}\label{indec} Let $R$ be analytically unramified. Then we can write the integral closure $\overline R$ as a finite product   $\overline R\cong \prod_{\p \in \Min(R)} \overline{R/\p}$, where each $\overline{R/\p}$ is an indecomposable finitely generated $R$-module.
\end{lem}    

\begin{proof} Since $R$ is reduced, so the decomposition follows by \cite[Corollary 2.1.13]{sh}. For each $\p\in \Min(R)$, $R/\p$ is analytically unramified by \cite[Tag 032Y]{st}. Hence $\overline{R/\p}$ is module finite over $R/\p$, so $\overline{R/\p}$ is module finite over $R$ as well. Now fix a $\p \in \Min(R)$ and let, $S=\overline{R/\p}$. Now the algebra $R\to S=\overline{R/\p}$ factors via $R\to R/\p \to \overline{R/\p}$, where the first map is surjective and the second map is a birational embedding. Since $S$ is a torsion-free $S$-module, so by Corollary \ref{homalgebra}, $S$ is indecomposable as an $S$-module if and only if  $S$ is indecomposable as an $R$-module. Now $S$ is an integral domain, so $S$ is indecomposable as an $S$-module. Hence the claim.   
\end{proof}   

We recall the following definition: 

\begin{dfn} (c.f. \cite{lindo}) For finitely generated modules $M$ over a commutative Noetherian ring $R$,
let us denote by $\tr_R (M)$ the ideal of $R$ defined as $\sum_{f\in \Hom_R(M,R)} f(M)$. We will refer to this as the trace ideal of $M$. When the ring in question is clear we will drop the suffix $R$ and may simply write $\tr_R(M)$. It is easy to see that, if $I$ is an ideal of $R$, then $I\subseteq \tr_R(I)$. 
\end{dfn}

\begin{prop}\label{traceann}  Let $M$ be a finitely generated module over a Noetherian local ring $R$. Then the following holds:  

\begin{enumerate}[\rm(1)]
    \item  $\ann_R   \Ext_R^1(M,\syz_R M)=\ann_R  \Ext_R^{>0}(M,\Mod R)$ 
    \item  If $M$ is a finitely generated $R$-module such that $M$ has a summand, which is isomorphic to a fractional ideal containing a non-zero-divisor of $R$, then $\ann_R\Ext^1_R(M,\syz_R M)\subseteq \tr_R(M)$.
\end{enumerate}   
\end{prop}   

\begin{proof} (1) This is \cite[Lemma 2.14]{iy}.
\\
(2) If $M$ is a finitely generated $R$-module such that $M\cong N\oplus X$, where $N$ is an $R$-submodule of $Q(R)$ containing a non-zero-divisor of $R$, then $\Ext^1_R(N,\syz_R N)$ is a direct summand of $\Ext^1_R(M,\syz_R M)$. So, in view of \cite[Theorem 1.1]{sde}, we get $\ann_R \Ext^1_R(M,\syz_R M)\subseteq \ann_R \Ext^1_R(N,\syz_R N) =\tr_R(N)\subseteq \tr_R(N)+\tr_R(X)=\tr_R(M)$. 
\end{proof}  

We end this section by recalling the notion of Ulrich modules over a local Cohen--Macaulay ring. Let $(R,\m,k)$ be a local Cohen--Macaulay ring. An $R$-module $M$ is called maximal Cohen--Macaulay (which we will also abbreviate as MCM hence forth) if $\depth_R M\ge \text{ dim } R$. So, in particular, we also consider $0$ to be MCM. We denote the full subcategory of MCM $R$-modules by $\cm(R)$.
By $\syz \cm(R)$, we will denote the subcategory consisting of all $R$-modules $M$, which fits into a short exact sequence $0\to M \to F \to N \to 0$, where $F$ is a free $R$-module and $N$ is an MCM $R$-module. By $\syz \cm^{\times}(R)$, we denote all modules in $\syz \cm(R)$, which does not have $R$ as a summand.    
An MCM $R$-module $M$ is called Ulrich if $e(M)=\mu(M)$. We denote the full subcategory of Ulrich $R$-modules by $\Ul(R)$. The following proposition is well known, and easily follows from \cite[Proposition A.19, A.23]{lw}.  

\begin{prop}\label{uldef} Let $(R,\m,k)$ be a local Cohen--Macaulay ring of dimension $d$. Let $M\in \cm(R)$. Then the following holds:  

\begin{enumerate}[\rm(1)]
\item  If $M$ is Ulrich, then $\m M=(x_1,...,x_d)M$ for every ideal $(x_1,...,x_d)\subseteq \m$, which forms a reduction of $\m$.

\item  If there exists an ideal $(x_1,...,x_d)\subseteq \m$, which forms a reduction of $\m$ and satisfies $\m M=(x_1,...,x_d)M$, then $M$ is Ulrich.  
\end{enumerate}  
\end{prop}  

We further record the following proposition, which will be used throughout the entire paper without further reference. 

\begin{prop} \label{reduction} Let $(R,\m,k)$ be a local Cohen--Macaulay ring, and let $M$ be an Ulrich $R$-module. If $x_1\in \m$ is part of a system of parameters, which forms a reduction of $\m$, then $M/x_1M$ is an Ulrich $R/x_1R$-module.  
\end{prop}

\begin{proof} Let $d=\dim R$, and let $x_1,...,x_d\in \m$ be a system of parameters, which forms a reduction of $\m$. Then $\m M=(x_1,...,x_d)M$ by Proposition \ref{uldef}(1). Let us denote $\overline{(-)}:=(-)\otimes_R R/x_1R$. Then $(\bar x_2, ..., \bar x_d)$ form a reduction of the maximal ideal $\overline \m=\m/(x_1)$ of $\overline R$ and $\overline \m \overline M=(\bar x_2, ..., \bar x_d) \overline M$. Since $\dim \bar R=d-1$, so by Proposition \ref{uldef}(2) we get, $\overline M$ is $\overline R$-Ulrich.  
\end{proof}

\begin{rem}\label{faithflat} 
Let $(R,\m)\to (S,\n)$ be a faithfully flat map such that $\n=\m S$. 

Let $M\in \mod (R)$. By \cite[Theorem 2.1.7]{bh} (and the sentences following it), we have $M\in \cm(R)$ if and only if $M\otimes_R S \in \cm(S)$. Let $I$ be an $\m$-primary ideal of $R$. As $\m S$ is the maximal ideal of $S$, so $IS=I\otimes_R S$ is also primary to the maximal ideal of $S$. Now $S\otimes_R \dfrac {M}{IM}\cong \dfrac{S\otimes_R M}{S\otimes_R (IM)}=\dfrac{S\otimes_R M}{(IS)(S\otimes_R M)}$ by \cite[Lemma 5.2.5]{ddd}. Hence $\lambda_S\left(\dfrac{S\otimes_R M}{(IS)(S\otimes_R M)}\right)=\lambda_S\left(S\otimes_R \dfrac {M}{IM} \right)=\lambda_R\left(\dfrac{M}{IM}\right)$, where the last equality is by \cite[Tag 02M1]{st} remembering that $S/\m S$ is the residue field of $S$. Also, since $e_R(I,M)=\lim_{n\to \infty}\dfrac{(\dim M)!}{n^{\dim M}}\lambda_R\left(\dfrac{M}{I^{n+1}M}\right)$, and $\dim_S(M\otimes_R S)=\dim_R(M)$ (see \cite[Theorem A.11(b)]{bh}), so a similar argument as the previous one shows $e_R(I,M)=e_S(IS, S\otimes_R M)$. Thus $M\in \Ul(R)$ if and only if $M\otimes_R S\in \Ul(S)$.   
\end{rem} 

We will freely use these above facts about Ulrich modules throughout the remainder of the paper without any further reference.

For a subcategory $\X \mod R$, let $\add_R \X$ denote the smallest subcategory of $\mod R$ containing $\X$, closed under direct summands of finite direct sums. It is clear that, $M\in \add_R \X$ if and only if $M\oplus N \cong \oplus_{i=1}^n X_i$, for some $N \in \mod R$, and $X_i \in \X$. 
Proposition \ref{syzind} and Proposition \ref{ulind} together proves Theorem \ref{thm1}($(1)\implies (2)$).   

\begin{prop}\label{syzind} Let $(R,\m,k)$ be a local ring of depth $t\ge 1$. Let $M$ be an $R$-module with $\depth M\ge t$. If $M\in \add_R\{R,\syz^i_R k: i\ge 0\}$, then $M\in \add_R\{R, \syz^i_R k: i\ge t\}$. 
\end{prop}

\begin{proof} There exists non-negative integers $b,a_0,...,a_n$ and $R$-module $N$ such that $M\oplus N\cong R^{\oplus b} \oplus (\oplus_{i=0}^n(\syz^i_Rk)^{\oplus a_i}$). Taking completion we get that, $\widehat M\oplus \widehat N \cong \widehat R^{\oplus b} \oplus (\oplus_{i=0}^n(\syz^i_{\widehat R}k)^{\oplus a_i})$. Now write $\widehat M\cong \oplus_{i=1}^l X_i, \widehat N\cong \oplus_{i=1}^s Y_i$, where $X_i, Y_i$ s are indecomposable $\widehat R$-modules. Since $\depth_{\widehat R}\widehat M=\depth_R M\ge t=\depth \widehat R$, so $\depth_{\widehat R} X_i\ge t$ for all $1\leq i\leq l$. Since for every $0\le j \le t-1$, we have $\depth_{\widehat R} \syz^j_{\widehat R}k=j<t$, so none of the $X_i$ s can be isomorphic to $\syz^j_{\widehat R}k$ for $0\le j \le t-1$. Since $\syz^j_{\widehat R}k$ is indecomposable for $0\le j\le t-1$ (by \cite[Coorllary 3.3]{TaG}), so by uniqueness of indecomposable decomposition in $\mod(\widehat R)$, we get that all the copies of $\syz^j_{\widehat R}k$ for $0\le j \le t-1$ must be isomorphic to one of the $Y_i$ s. Hence by \cite[Corollary 1.16]{lw}, we can cancel all the copies of $\syz^j_{\widehat R}k$ for $0\le j \le t-1$ from both sides of $\widehat M\oplus \widehat N \cong \widehat R^{\oplus b} \oplus (\oplus_{i=0}^n(\syz^i_{\widehat R}k)^{\oplus a_i})$. Thus we get that, $\widehat M \in \add_{\widehat R}\{\widehat R,\syz^j_{\widehat R}k: n\geq j\ge t\}\subseteq \add_{\widehat R}(\reallywidehat{R\oplus(\oplus_{i=t}^n\syz^i_{R}k}))$. Hence $M\in \add_{ R}(R\oplus(\oplus_{i=t}^n\syz^i_{R}k))$ by \cite[Proposition 2.18]{lw}. Thus, $M\in \add_R\{R, \syz^i_R k: i\ge t\}$. 
\end{proof} 

\begin{prop}\label{ulind} Let $(R,\m)$ be a local Cohen--Macaulay ring. Let $\X \subseteq \cm(R)$ be a subcategory containing some non-zero module.  Then $\Ul(R)\subseteq \add_R \X$ if and only if $\Ul(R) \subseteq \add_R\{R,\X\}$. 
\end{prop}   

\begin{proof} We only need to prove that, if $\Ul(R) \subseteq \add_R\{R,\X\}$, then $\Ul(R)\subseteq \add_R \X$. Firstly, if $R$ is regular, then $\Ul(R)=\add_R(R)=\cm(R)=\add_R\X$ (where the last equality holds since $\X$ has a non-zero module, so $R\in \add_R \X$). Hence there is nothing to prove. So, now we assume $R$ is singular, which implies $R\notin \Ul(R)$. Now suppose $\Ul(R) \subseteq \add_R\{R,\X\}$. Let $M\in \Ul(R)$ (so, $R$ is not a direct summand of $M$). Then $M\oplus N\cong R^{\oplus b} \oplus (\oplus_i X_i)$, where $N$ is an $R$-module and $X_i\in \X$, and $b\ge 0$ is an integer. Now write $N \cong N'\oplus R^{\oplus a}$, where $N'$ does not have $R$ as a summand. So, we have $M\oplus N' \oplus R^{\oplus a}\cong R^{\oplus b} \oplus (\oplus_i X_i)$. If $b>a$, then by \cite[Corollary 1.16]{lw} we get, $M\oplus N' \cong R^{\oplus b-a}  \oplus (\oplus_i X_i)$. Hence $R$ is a direct summand of either $M$, or $N'$ by \cite[Lemma 1.2(i)]{lw}. But this contradicts that, neither $M$ nor $N'$ has $R$ as a direct summand. Thus $b\le a$. Hence \cite[Corollary 1.16]{lw} again yields $M\oplus N' \oplus R^{\oplus a-b}\cong \oplus_i X_i$. Hence $M\in \add_R \X$.      
\end{proof}

\begin{chunk}\label{1split} Let $\X\subseteq \mod R$ be an additively closed subcategory (closed under direct summands of finite direct sums), where $R$ is a commutative Noetherian ring. We denote by $\ind \X$ the collection of indecomposable objects (up to isomorphism) of $\X$. If $R$ is local and $\X$ has exactly one indecomposable object, then we observe that every short exact sequence of modules in $\X$ splits. Indeed, let $X$ be the unique indecomposable object in $\X$. Since $\X$ is additively closed, so for every $M\in \X$, every direct summand of $M$ is in $\X$. Hence  we can write $M\cong X^{\oplus a}$ for some $a\ge 0$. So, every short exact sequence of modules in $\X$ looks like $0\to X^{\oplus a} \to X^{\oplus b} \to X^{\oplus c} \to 0$ for some integers $a,b,c\ge 0$. Localizing this sequence at a minimal prime of $X$, and then computing length, we see that $b=a+c$. Hence $X^{\oplus b}\cong X^{\oplus a}\oplus X^{\oplus c}$. Thus by Miyata's Theorem (see \cite[Theorem 7.1]{lw} for example), the sequence splits.
\end{chunk} 

In view of this, the next non-trivial case is when $\X$ has exactly two indecomposable objects. Interestingly, even this next case, when the number of indecomposable objects in $\X$ is very small, is quite non-trivial, and we can only say something worthwhile under some additional hypothesis. In what follows, $G(R)$ denotes the Grothendieck group of the category of finitely generated $R$-modules. 

\begin{prop}\label{Ggroup} Let $R$ be a Henselian local ring, and $\X\subseteq \mod R$ be an additively closed subcategory such that $\X$ along with all short exact sequences of objects in $\X$ is an exact subcategory of $\mod R$ (e.g., $\X=\Ul(R)$, see \cite[Corollary 4.7]{ddd}). Assume $\operatorname{ind} \X=\{A,B\}$. If $A,B$ both have same constant positive rank and $2[A]\ne 2[B]$ in $G(R)$, then every short exact sequence of modules in $\X$ splits.  
\end{prop}  

\begin{proof} In the notation of \cite[4.9]{ddd} we see that, $\Ext^1_{\X}(-,-):\X^{op}\times \X\to \mod R$ is a subfunctor of $\Ext^1_R(-,-)|_{\X^{op}\times \X}:\X^{op}\times \X \to \mod R$ by \cite[Proposition 3.8]{ddd}. Since by Krull–Schmidt theorem, every module in $\X$ can be written as a finite direct sum of indecomposable modules in $\X$, hence due to additivity of the subfunctor in each component, it is enough to prove that $\Ext^1_{\X}(M,N)=0$ for each indecomposable object $M,N$ in $\X$, i.e., $\Ext^1_{\X}(A,B)=\Ext^1_{\X}(B,A)=\Ext^1_{\X}(A,A)=\Ext^1_{\X}(B,B)=0$.
Due to symmetric nature of $A,B$, it is enough to prove $\Ext^1_{\X}(A,B)=\Ext^1_{\X}(A,A)=0$. Let $r>0$ denote the common rank of $A,B$.  

First we prove that, $\Ext^1_{\X}(A,B)=0$. Indeed, let $0\to B \to A^{\oplus a}\oplus B^{\oplus b} \to A \to 0$ be a short exact sequence in $\Ext^1_{\X}(A,B)$. Comparing ranks, we have $ra+rb=r+r$, so $a+b=2$. Both the possibilities $a=0, b=2$ and $a=2, b=0$ lead to $[A]=[B]$ in $G(R)$, contradiction! Thus $a=b=1$, and then the sequence splits by Miyata's Theorem (see \cite[Theorem 7.1]{lw}). 

Next we prove that, $\Ext^1_{\X}(A,A)=0$. Indeed, let $0\to A \to A^{\oplus a'}\oplus B^{\oplus b'} \to A \to 0$ be a short exact sequence in $\Ext^1_{\X}(A,A)$. Again, by comparing ranks we get, $a'+b'=2$. Now $a'=b'=1$ leads to $[A]=[B]$ in $G(R)$, impossible! Also, $a'=0, b'=2$ leads to $2[B]=2[A]$ in $G(R)$, impossible! Thus the only remaining choice is $a'=2, b'=0$, in which case the sequence splits by Miyata's Theorem (see \cite[Theorem 7.1]{lw}), as we wanted.  
\end{proof}  

\begin{rem}\label{remark2.15} In the above proposition, if $\X$ were not an exact subcategory of $\mod R$, then given modules $M,N\in \X$, the collection of all short exact sequences in $\Ext^1_R(M,N)$ with objects in $\X$ might not have been a submodule of $\Ext^1_R(M,N)$. Hence we would not get a subfunctor of $\Ext^1_R(-,-)$. Then we would have to deal with general short exact sequences of the form $0\to A^{\oplus a}\oplus B^{\oplus b} \to A^{\oplus a'}\oplus B^{\oplus b'}\to A^{\oplus a''}\oplus B^{\oplus b''}\to 0$, instead of only those whose end terms are the indecomposable objects $A,B$. Then in the hypothesis, instead of only requiring $2[A]\ne 2[B]$ in $G(R)$, we would have to assume the condition that \textit{if $m[A]+n[B]=0$ in $G(R)$, then $m=n=0$ (i.e., $[A],[B]$ are $\mathbb Z$-linearly independent in $G(R)$)}, and this could be harder to verify in practice.  
\end{rem}   
  
\section{Short exact sequences of Ulrich modules}

In this section, we focus on splitting of short exact sequences of Ulrich modules for any local Cohen--Macaulay ring. 

We begin with a small remark, which will be needed for our next results.  

\begin{rem}\label{syzflat} If $(R,\m) \to (S,\n)$ is a flat homomorphism of local rings, then for every finite $R$-module $M$, it holds that $S\otimes_R \syz_R M\cong \syz_S(S\otimes_R M)$. Indeed, tensoring the beginning of a minimal $R$-free resolution $R^{\oplus b} \xrightarrow{f} R^{\oplus a}\to M \to 0$ with $S$, we get an exact sequence $S^{\oplus b} \xrightarrow{f\otimes_R \text{id}_S} S^{\oplus a}\to S\otimes_R M \to 0$. Since $Im(f)\subseteq \m^{\oplus a}$, so $Im(f\otimes_R \text{id}_S)\subseteq (\m \otimes_R S)^{\oplus a}=(\m S)^{\oplus a}\subseteq \n^{\oplus a}$ (where $\m \otimes_R S=\m S$, since $S$ is flat). Thus $S^{\oplus b} \xrightarrow{f\otimes_R \text{id}_S} S^{\oplus a}\to S\otimes_R M \to 0$ is the beginning of a minimal $S$-free resolution of $S\otimes_R M$. Hence $S\otimes_R \syz_R M\cong \syz_S(S\otimes_R M)$.  
\end{rem}

The result below proves Theorem \ref{thm1}((2)$\implies$(3)).

\begin{thm}\label{uladdsplit} Let $(R,\m,k)$ be a local Cohen--Macaulay ring of dimension $d$. If $0\to N \to L \to M\to 0$ is a short exact sequence of Ulrich modules, where $M\in \add_R \{R,\syz^i_Rk: i\ge 0\}$, then the exact sequence splits. So, in particular, if $\Ul(R)\subseteq \add_R \{R,\syz^i_Rk: i\ge 0\}$, then $R$ is $\us$. 
\end{thm}       
    
\begin{proof} 


Consider the faithfully flat extension $S=R[X]_{\m[X]}$ with unique maximal ideal $\n=\m S$ and infinite residue field $S/\m S=k\otimes_R S$. Since $M\in \add_R \{R,\syz^i_Rk: i\ge 0\}$, so $S \otimes_R M \in \add_S\{S,S\otimes_R\syz^i_Rk: i\ge 0\}=\add_S\{S,\syz_S^i(k\otimes_R S): i\ge 0\}$ (using Remark \ref{syzflat}). Moreover, since $S$ is faithfully flat, so the short exact sequence  $0\to N \to L \to M\to 0$ splits if and only if the short exact sequence $0 \to N \otimes_R S \to L\otimes_R S \to M \otimes_R S\to 0$ splits.  Thus, without loss of generality, we can now assume that $R$ has infinite residue field $k$.   

Now we prove the claim by induction on $d$. The claim is obvious if $d=0$, since everything in $\Ul(R)$ is a $k$-vector space. Now for the inductive step, assume that the claim is true for all rings of dimension $d-1$. Let $\dim R=d\ge 1$. Since $k$ is infinite, we can choose a parameter ideal $(x_1,\cdots,x_d)$, which is a reduction of $\m$, such that $x_1\notin \m^2$. Then $R/(x_1)$ has dimension $d-1$. Let us denote $\overline {(-)}:=(-)\otimes_R R/(x_1)$. Let $\sigma: 0\to N \to L \to M\to 0$ be a short exact sequence, where $M,N,L\in \Ul(R)$ and $M\in \add_R \{R,\syz^i_Rk: i\ge 0\}$. Then $M/x_1M, N/x_1N,L/x_1L\in \Ul(\overline R)$. Since $M\in \add_R \{R,\syz^i_Rk: i\ge 0\}$, so $M\oplus X \cong R^{\oplus b} \oplus \left( \oplus_{j=0}^{n} (\syz^j_R k)^{\oplus q_j} \right)$ for some $R$-module $X$ and integers $b,n,q_0,...,q_{n}\ge 0$.  Applying $\overline{(-)}$ and using \cite[Corollary 5.3]{TaG}, we have $\overline M \oplus \overline X \cong (\overline R)^{\oplus b} \oplus k^{\oplus q_0}\oplus \left(\oplus_{j=1}^{n}(\syz^j_{\overline R}k \oplus \syz^{j-1}_{\overline R} k)^{\oplus q_j} \right)$. 
Hence $\overline M\in \add_{\overline R}\{\overline R, \syz^{j}_{\overline R} k, j\ge 0\}$. So, by induction hypothesis, the short exact sequence $\sigma \otimes_R R/(x_1)$ (this is short exact as $x_1$ is $M$-regular) splits. Since $x_1$ is $R,M,N,L$-regular, so $\sigma \in x_1\cdot \Ext^1_R(M,N)$ by \cite[Proposition 2.8]{jan}. Now as $M\in \add_R \{R,\syz^i_Rk: i\ge 0\}$, so $\Ext^1_R(M,N)$ is killed by $\m$. Thus $\sigma \in  x_1\cdot \Ext^1_R(M,N) \subseteq \m \cdot \Ext^1_R(M,N)=0$, so $\sigma$ splits.  This finishes the proof by induction.   
\end{proof}

We end this section with the following proposition, which will play a crucial role in the next section for characterizing rings of minimal multiplicity that are $\us$.  

\begin{prop}\label{addresi} Let $(R,\m,k)$ be a local Cohen--Macaulay ring of dimension $d$. Let $M$ be an Ulrich $R$-module such that $\m\Ext^1_R(M,\syz_R M)=0$. Then $M\in \add_R(\syz^d_R k)$.    
\end{prop} 

\begin{proof} Consider the faithfully flat extension $S=R[X]_{\m[X]}$ with infinite residue field $S\otimes_Rk$ and unique maximal ideal $\m S$. Applying Remark \ref{syzflat} and \cite[Lemma 5.2.5(1)]{ddd}, we get the following chain of isomorphisms:

\begin{align*}
(\m S)\Ext^1_S(S\otimes_R M,\syz_S(S\otimes_R M))&\cong (\m S)\Ext^1_S(S\otimes_R M,S\otimes_R \syz_R M)\\
&\cong (\m S)(S\otimes_R\Ext^1_R(M,\syz_R M))\cong S\otimes_R (\m \Ext^1_R(M,\syz_R M))=0
\end{align*} 

Now let, $N:=S\otimes_R M$ and $l:=S\otimes_R k$ (this is the residue field of $S$). Then $N$ is an Ulrich $S$-module by Remark \ref{faithflat}. Since $S$ has infinite residue field, so we can choose $(x_1,...,x_d)\subseteq \m S$ such that $(x_1,...,x_d)$ is a reduction of $\m S$. Then $N/(x_1,...,x_d)N$ is a $l$-vector space, i.e., $N/(x_1,...,x_d)N\cong l^{\oplus a}$ for some $a\ge 0$. By \cite[Lemma 2.14]{iy} we get, $(x_1,...,x_d)\subseteq (\m S) \subseteq \ann_S \Ext^1_S(N, \syz_S N)=\ann_S \Ext^{\ge 1}_S (N, \Mod S)$. Thus from \cite[Proposition 2.2]{stable} we get, $(\syz^d_S l)^{\oplus a}\cong \bigoplus_{i=0}^d (\syz_S^i N)^{\oplus b_i}$, where  $b_i=\binom{d}{i}$ for all $0\leq i\leq d$. Remembering $l=S\otimes_R k$ and $N=S\otimes_R M$ we get that, $S\otimes_R(\syz_R^d k)^{\oplus a}\cong S\otimes_R \bigoplus_{i=0}^d (\syz_R^i M)^{\oplus b_i}$ by using Remark \ref{syzflat}. So, by \cite[Proposition 2.5.8]{Gro} we get, $(\syz_R^d k)^{\oplus a}\cong \bigoplus_{i=0}^d (\syz_R^i M)^{\oplus b_i}$. Since $b_0=1$, hence $M\in \add_R(\syz_R^d k)$. 
\end{proof}

\section{Rings of minimal multiplicity}

Let $R$ be a local Cohen--Macaulay ring and $M\in \cm(R)$. It follows from \cite[Lemma 2.2]{del} that, $\syz_R M \in \syz \cm^{\times}(R)$. So, if moreover $R$ has minimal multiplicity, then from \cite[Proposition 3.6]{umm} it follows that $\syz_R M\in \Ul(R)$. It is also known that $\syz^{\dim R}_R k \in \syz \cm^{\times}(R)$ when $R$ is singular, Cohen--Macaulay and has minimal multiplicity (see \cite[Lemma 5.7]{umm}). We will use these facts throughout the rest of the paper without possibly referencing further. In the following, we will also denote by $\cm_0(R)$ the subcategory of $\cm(R)$ consisting of modules locally free at all non-maximal prime ideals of $R$.   

Combining the results from the previous section, we get the following characterization of local Cohen--Macaulay rings of minimal multiplicity that are $\us$.

\begin{thm} \label{minuss} Let $(R,\m,k)$ be a local Cohen--Macaulay ring of dimension $d$. Then the following are equivalent:  

\begin{enumerate}[\rm(1)]
    \item  $\Ul(R)=\add_R(\syz_R^d k)$.
    
    \item  $R$ is $\us$ and has minimal multiplicity. 
    
    \item  $R$ has minimal multiplicity and $\m\Ext^1_R(M,N)=0$ for all $M,N\in \Ul(R)$. 
    
    \item $R$ has minimal multiplicity and $\m \Ext^i_R(M,\mod R)=0$ for all $M\in \Ul(R)$, and for all $i\ge 1$. 
    
If, moreover, $R$ has a canonical module, then the above are also equivalent to
    
    \item $R$ has minimal multiplicity and $\m \Ext^{>0}_R(\cm(R),\mod R)=0$
\end{enumerate}

If one (hence all) of the conditions (1) through (4) holds, then  $\widehat R$ is an isolated singularity.   
\end{thm} 

\begin{proof}  (1) $\implies (2)$: From Theorem \ref{uladdsplit} it follows that, $R$ is $\us$. Also, from \cite[Proposition 2.5]{ul} it follows that, $R$ has minimal multiplicity. 

(2) $\implies$ (3): This follows from \cite[Lemma 5.2.2]{ddd}. 

(3) $\implies$ (4): Since $R$ has minimal multiplicity, so $\syz_R M\in \Ul(R)$ for all $M\in \Ul(R)$ by \cite[Proposition 3.6]{umm}. Hence $\m \Ext^1_R(M,\syz_R M)=0$ for all $M\in \Ul(R)$. So, $\m \Ext^i_R(M,\mod R)=0$ for all $M\in \Ul(R)$, and for all $i\ge 1$ by \cite[Lemma 2.14]{iy}.

$(4) \implies (1)$: From Proposition \ref{addresi} we have, $\Ul(R)\subseteq\add_R(\syz_R^d k)$. Since $R$ has minimal multiplicity, so by \cite[Lemma 5.7(1)]{umm} we get that $\syz_R^d k\in \Ul (R)$. Thus $\Ul(R)=\add_R(\syz_R^d k)$.

This shows the equivalence of (1) through (4). Now assume $R$ has a canonical module $\omega$, and let $(-)^{\dagger}$ denote canonical dual. 

$(4)\implies (5)$: By \cite[Lemma 2.14]{iy}, it is enough to show that, $\m\Ext^1_R(M,\syz_R M)=0$ for all $M\in \cm(R)$. We know that, $\Ext^1_R(M,\syz_R M)\cong \Ext^1_R((\syz_R M)^{\dagger},M^{\dagger})$. Also, $\syz_R M$ is Ulrich since $R$ has minimal multiplicity. Hence $(\syz_R M)^{\dagger}$ is also Ulrich by \cite[Corollary 4.2]{umm}. Thus the assumption of (4) implies $0=\m \Ext^1_R((\syz_R M)^{\dagger},M^{\dagger})=\m\Ext^1_R(M,\syz_R M)$. 

(5) $\implies (1)$:  From Proposition \ref{addresi} we have, $\Ul(R)\subseteq\add_R(\syz_R^d k)$. Since $R$ has minimal multiplicity, so by \cite[Lemma 5.7(1)]{umm} we get that $\syz_R^d k\in \Ul (R)$. Thus $\Ul(R)=\add_R(\syz_R^d k)$.

This proves the equivalence of (1) through (5) when $R$ has a canonical module.  

Now we will prove that, $\widehat R$ is an isolated singularity when $R$ satisfies (1) through (4). Since $R$ having minimal multiplicity implies $\syz \cm^{\times}(R)\subseteq \Ul(R)$, so it is enough to prove that, if $\m\Ext^1_R(\syz \cm^{\times}(R),\syz \cm^{\times}(R))=0$, then $\widehat R$ is an isolated singularity. Now to prove this, note that if $M\in \cm(R)$, then $\syz_R M, \syz^2_R M \in \syz \cm^{\times}(R)$ by \cite[Lemma 2.2]{del}. Hence the hypothesis implies that $\m\Ext^1_R(\syz_R M,\syz^2_R M)=0$, so \cite[Lemma 2.14]{iy}  implies $\m \Ext^{\ge 1}_R(\syz_R M,\mod R)=0$, so $\m\Ext_R^{\ge 2}(M,\mod R)=0$. Hence we get $\m\Ext^{\ge 2}_R(\cm(R),\cm(R))=0$. So, in particular, $\m\Ext_R^{\ge 2}(\cm_0(R),\cm_0(R))=0$.  A same argument as done in the proof of \cite[Proposition 4.9(2)]{deyT} now shows that $\widehat R$ has isolated singularity. 
\end{proof}   

\if
Next we observe that local Cohen--Macaulay rings that are $\us$ have completion to be isolated singularity.  

\begin{prop}\label{compliso} Let $(R,\m)$ be a local Cohen--Macaulay ring (not necessarily of minimal multiplicity). If $\m\Ext^1_R(\syz \cm^{\times}(R),\syz \cm^{\times}(R))=0$, then $\widehat R$ is an isolated singularity. So, in particular, if $R$ is $\us$ and has minimal multiplicity, then $\widehat R$ is an isolated singularity.  
\end{prop} 

\begin{proof}  The second statement follows from the first since $R$ is $\us$ implies $\m\Ext_R^1(\Ul(R),\Ul(R))=0$ (\cite[Lemma 5.2.2]{ddd}), and $R$ minimal multiplicity implies $\syz \cm^{\times}(R)\subseteq \Ul(R)$ (\cite[Proposition 3.6]{umm}).   

To prove the first statement, note that by hypothesis, for every $M\in \cm(R)$, we have $\m\Ext^1_R(\syz_R M, \syz_R^2 M)=0$, hence \cite[Lemma 2.14]{iy}  implies $\m \Ext_R^{i\ge 1}(\syz_R M,\mod R)=0$, so $\m\Ext_R^{i\ge 2}(M,\mod R)=0$. So we get $\m\Ext_R^{i\ge 2}(\cm(R),\cm(R))=0$. A same argument as done in the proof of \cite[Proposition 4.9(2)(ii)]{deyT} using \cite[Lemma 3.11(3)]{deyT} now shows $\widehat R$ is an isolated singularity.  
\end{proof} 
\fi 


Theorem \ref{minuss}, in particular, also implies that $\us$ property ascends to completion for Henselian local Cohen--Macaulay rings of minimal multiplicity.  

\begin{cor}\label{hens} Let $(R,\m)$ be a Henselian local Cohen--Macaulay ring of minimal multiplicity. If $R$ is $\us$, then so is $\widehat R$.  
\end{cor}

\begin{proof} By Theorem \ref{minuss}, $\widehat R$ has isolated singularity. By \cite[Proposition 4.7]{auac}, we have that for every $M, N \in \Ul(\widehat R)$, there exist $X,Y\in \cm(R)$ such that $M\cong \widehat X, N\cong \widehat Y$, and consequently
$X,Y\in \Ul(R)$. Since $R$ is $\us$, so $\m \Ext^1_R(X,Y)=0$ by \cite[Lemma 5.2.2]{ddd}. Hence by taking completion we get that, $0=\widehat \m \reallywidehat{\Ext^1_{ R}( X, Y)}\cong \widehat \m \Ext^1_{\widehat R}(\widehat X, \widehat Y)\cong \widehat \m \Ext^1_{\widehat R}(M,N)$. Since this is true for any $M,N\in \Ul(\widehat R)$, and since $\widehat R$ also has minimal multiplicity, we get that $\widehat R$ is $\us$ by Theorem \ref{minuss}.   
\end{proof}

Next, we use Theorem \ref{minuss} to deduce information about Gorenstein local rings of minimal multiplicity (i.e., abstract hypersurfaces of minimal multiplicity), which are $\us$.

First, we record the following lemma.

\begin{lem}\label{hypsur} Let $R$ be a singular local Gorenstein ring of minimal multiplicity. Then the following holds:
    
    \begin{enumerate}[\rm(1)]
        \item Every maximal Cohen--Macaulay module without a non-zero free summand is Ulrich. \item Every maximal Cohen--Macaulay module is a direct sum of an Ulrich module and a free module.  
    \end{enumerate}
    \end{lem}
    
    \begin{proof} Let $d=\dim R$.
    
    (1) Let $M$ be a maximal Cohen--Macaulay module without a non-zero free summand.  Since $M$ is totally reflexive, so it is a $(d+1)$-th syzygy module. Hence $M\cong F\oplus \syz_R^{d+1}N$ for some $R$-module $N$ and some free $R$-module $F$. Since $M$ has no free summand, so $M\cong \syz_R(\syz_R^d N)\in \syz \cm^{\times}(R)\subseteq\Ul(R)$ by \cite[Proposition 3.6]{umm}. Hence $M$ is Ulrich. 
    
    (2)  This follows from an isomorphism $M\cong F\oplus \syz_R^{d+1}N$ as in (1).  
    \end{proof}
    
As a consequence, we get the following characterization of local Gorenstein rings of minimal multiplicity, which are $\us$. 
    
\begin{prop}\label{hyp} Let $R$ be a local Gorenstein ring of minimal multiplicity. Let $\dim R=d$ and $k$ be the residue field of $R$. Then the following are equivalent:  
    
    \begin{enumerate}[\rm(1)]
        \item  $R$ is $\us$. 
        \item $\cm(R)=\add_R (R \oplus \syz^d_R k)$. 
    \end{enumerate}   
    
    So, in particular, if either of the above equivalent conditions holds, then $R$ has finite CM type.  
    \end{prop} 
    
    \begin{proof}  There is nothing to prove if $R$ is regular, so assume $R$ is singular.
    
    $(1)\implies (2)$: From Lemma \ref{hypsur}(2), it follows that every $M\in \cm(R)$ can be written as $M\cong N\oplus F$ for some $N\in \Ul(R)$ and free $R$-module $F$. Hence the claim follows from Theorem \ref{minuss}.  
    
    $(2) \implies (1)$: This follows from  Theorem \ref{uladdsplit}.  
    \newline 
 That $R$ has finite CM type follows from \cite[Theorem 2.2]{lw}.  
    \end{proof}  
    
Similar to Corollary \ref{hens}, we also get the following corollary. 

\begin{cor}\label{gorcompl} Let $R$ be a local Gorenstein ring of minimal multiplicity. If $R$ is $\us$, then so is $\widehat R$. 
\end{cor}

\begin{proof}  Let $d=\dim R$. By Theorem \ref{minuss}, $\widehat R$ has isolated singularity. Hence for every $N\in \cm(\widehat R)=\cm_0(\widehat R)$, there exists $M\in \cm(R)$ such that $N$ is a direct summand of $\widehat M$ (\cite[Corollary 3.3]{stable}). Since $\cm(R)=\add_R (R \oplus \syz^d_R k)$ by Proposition \ref{hyp}, so we get that $\cm(\widehat R)\subseteq \add_{\widehat R}(\widehat R \oplus \syz^d_{\widehat R}k)$. Hence $\widehat R$ is $\us$ by Theorem \ref{uladdsplit}.    
\end{proof}  

In the next proposition, we record how $\us$ property behaves under passing to double branched cover of hypersurfaces.  
    
    \begin{prop}\label{branch} Let $R=S/(f)$ be a singular hypersurface (where $S$ is a regular local ring) of dimension $d>0$, and denote $R^{\#}:=S[[z]]/(f+z^2)$ (note that, $R^{\#}$ always has minimal multiplicity). Then the following hold:   
    \begin{enumerate}[\rm(1)] 
        \item Assume $2$ is a unit in $S$. If $R$ is $\us$ and has minimal multiplicity, then $R^{\#}$ is also $\us$.  
        
        \item  If $R^{\#}$ is $\us$, then $R$ is $\us$. 
    \end{enumerate} 
    \end{prop}    
    
    \begin{proof} Let $\dim R=d$, so $\dim R^{\#}=d+1$. For an $R^{\#}$-module $M$, let us denote $\overline {M}:=M \otimes_{R^{\#}} R^{\#}/(z)$. Note that, $\overline{R^{\#}}=R$ and $\syz_{R^{\#}} R\cong R^{\#}$. Let $k$ be the common residue field of $R$ and $R^{\#}$.          
    
    (1) Let $M \in \cm (R^{\#})$, so $\overline M\in \cm(R)$. 
    Since $R$ is $\us$ and has minimal multiplicity, so by Proposition \ref{hyp} we have $\overline M\oplus X\cong R^{\oplus a} \oplus (\syz^d_R k)^{\oplus a}$ for some $X\in\mod(R)$ and integer $a\ge 0$. So, $\syz_{R^{\#}} (\overline M) \oplus Y \cong (R^{\#})^{\oplus a}\oplus (\syz_{R^{\#}}\syz^d_R k)^{\oplus a}$ for some $R^{\#}$-module $Y$. By \cite[Proposition 2.5]{branched} we have, $\syz_{R^{\#}} (\overline M)\cong M \oplus \syz_{R^{\#}}  M$, and by \cite[Lemma 4.3]{branched} we have, $F_0\oplus \syz_{R^{\#}}\syz^d_R k\cong \syz^{d+1}_{R^{\#}} k \oplus F_1$ for some free $R^{\#}$-modules $F_0,F_1$.  Hence $M\oplus Z \cong (R^{\#})^{\oplus b} \oplus (\syz^{d+1}_{R^{\#}} k)^{\oplus a'}$ for some $R^{\#}$-module $Z$ and integers $a',b\ge 0$. Hence $\cm(R^{\#})=\add_{R^{\#}}\{R^{\#}, \syz_{R^{\#}}^{d+1}k\}$, so $R^{\#}$ is $\us$ by Theorem \ref{uladdsplit}.   
    
    (2) By Corollary \ref{gorcompl} we can assume that $\widehat {R^{\#}}$ is $\us$, so we can now work with $S$ complete. For an $R$-module $M$, let $M^{\#}:=\syz_{R^{\#}} M$. Let $M\in \Ul(R)$. Then $M^{\#}\in \cm(R^{\#})$. Since $R^{\#}$ is $\us$ and has minimal multiplicity, so by Proposition \ref{hyp} we have, $M^{\#}\oplus X \cong (R^{\#})^{\oplus a} \oplus (\syz^{d+1}_{R^{\#}}k)^{\oplus a}$ for some $R^{\#}$-module $X$ and integer $a\ge 0$. So, $\overline{M^{\#}} \oplus \overline X \cong (\overline{R^{\#}})^{\oplus a} \oplus (\overline{\syz^{d+1}_{R^{\#}}k})^{\oplus a}$. By \cite[Proposition 8.15]{lw} we have  $\overline{M^{\#}}\cong M \oplus \syz_R M$, and by  \cite[Corollary 5.3]{TaG} we have $\overline{\syz^{d+1}_{R^{\#}}k}\cong \syz^d_{\overline{R^{\#}}} k \oplus \syz^{d+1}_{\overline{R^{\#}}} k \cong \syz^d_R k \oplus \syz^{d+1}_R k$. Hence we get $M \oplus Y \cong R^{\oplus a} \oplus (\syz^d_R k \oplus \syz^{d+1}_R k)^{\oplus a}$ for some $R$-module $Y$.  Hence we have shown $\Ul(R) \subseteq \add_R\{R,\syz^d_R k,\syz_R^{d+1}k\}$. So, by Theorem \ref{uladdsplit} we get $R$ is $\us$.         
    \end{proof}

We finish this section by fully describing complete local Cohen--Macaulay rings $(R,\m,\mathbb C)$ of dimension $2$ and minimal multiplicity, which are $\us$, namely, they are only the cyclic quotient singularities with $|\text{ind}\Ul(R)|\le 2$.  Here, by cyclic quotient singularity, we mean a local ring of the form $k[[x,y]]^G$, where $k$ is algebraically closed, and $G=\left\langle\sigma=\begin{pmatrix}
\zeta_n & 0\\ 0 & \zeta^a_n
\end{pmatrix} \right\rangle \subseteq GL(2,k)$, where $n$ is invertible in $k$, $\zeta_n$ is a primitive $n$-th root of unity, and $1\le a\le n-1$ and gcd$(a,n)=1$, and we denote this cyclic group by $\dfrac 1 n(1,a)$ following the notation of \cite[Section 2]{cyc}.  In the following, we freely use the notations and basic preliminaries as in \cite[Section 2]{cyc}. 

First, we show that cyclic quotient surface singularities, which have exactly two indecomposable Ulrich modules, must be $\us$. Note that, for a cyclic  quotient surface singularity $R=S^G$, where $S=k[[x,y]]$, MCM $R$-modules are same as reflexive $R$-modules, and every indecomposable MCM $R$-module is of the form $M :=(S\otimes_k V)^G$ for some irreducible representation $V$ of $G$. First, we need some preparatory results.  We also use the description of divisor class group as in \cite[Section 2.10]{sg}.   
 
 \begin{lem}\label{quotcl} Let $R=S^G$ be a cyclic  quotient surface singularity with divisor class group $Cl(R)$. Let $|G|=n$. Then $\mathbb Z/(n) \xrightarrow{t \mapsto [M_t]} Cl(R)$  (where $0\le t\le n-1$) gives an isomorphism. 
 \end{lem}   
 
 \begin{proof} In this proof, we use the notations and preliminaries as in \cite[Section 2]{cyc}. Clearly, the map is injective, as for distinct values of $0\le t\le n-1$, $M_t$ s are non-isomorphic. Now we first show that the map is a group homomorphism, i.e., $[M_{i+j}]=[M_i]+[M_j]$, i.e., $M_{i+j}\cong (M_i\otimes_R M_j)^{**}$.  By \cite[Theorem 3.4(4)]{bc}, it is then enough to show that $V_i \otimes_k V_j\cong V_{i+j}$, where the isomorphism is as $k[G]$-modules (i.e., a $G$-equivariant isomorphism of the corresponding representations). To prove this later isomorphism, first note that since $G$ is cyclic, so irreducible representations of $G$ are one-dimensional. Hence for all $i$ we have, $V_i=k$ as $k$-vector spaces. Hence we have a $k$-linear map $T: V_i\otimes_k V_j \xrightarrow{b \otimes c \mapsto bc} V_{i+j}$. Now using the definition of tensor product of representations, we easily check that $T(\sigma \cdot (b\otimes c))=T((\sigma.b)\otimes (\sigma.c))=T(\zeta_n^{-i}b\otimes\zeta_n^{-j}c)=T(\zeta_n^{-(i+j)}(b\otimes c))=\zeta_n^{-(i+j)}T(b\otimes c)=\zeta_n^{-(i+j)}bc=\sigma \cdot T(b\otimes c)$ for all elementary tensors $b\otimes c\in V_i\otimes_k V_j$. Since $G$ is generated by $\sigma$, so this proves that $T: V_i\otimes_k V_j \xrightarrow{b \otimes c \mapsto bc} V_{i+j}$ is $G$-equivariant. Since $V_i\otimes_k V_j, V_{i+j}$ are all one-dimensional vector spaces, so $T$ is an isomorphism. Thus we get a $G$-equivariant isomorphism $V_i \otimes_k V_j\cong V_{i+j}$ as required. Finally, our map $\mathbb Z/(n) \xrightarrow{t \mapsto [M_t]} Cl(R)$ is surjective, since reflexive modules of rank $1$ over $R$ are indecomposable and any such module is isomorphic to some $M_t:=(S\otimes_k V_t)^G$, where $V_t=k$ is the one dimensional $k$-vector space equipped with the representation $\rho_t: G \xrightarrow{\sigma \mapsto \zeta_n^{-t}} k^{\times}$.   
 \end{proof}
 
 \begin{prop}\label{1ul} The cyclic quotient surface singularity $R=S^G$, where $G=\dfrac 1 n (1,a)$, has exactly one indecomposable Ulrich module if and only if $a=1$, i.e., $R$ is the $n$-th Veronese subring of $k[[x,y]]$.
 \end{prop}
 
 \begin{proof} By \cite[Theorem 1.5]{cyc}, $R$ has exactly one indecomposable Ulrich module if and only if $r=1$ (where $r$ denotes the number of irreducible exceptional curve, see \cite[Theorem 1.5, Theorem 2.4]{cyc}). This by \cite[Theorem 2.4]{cyc} exactly means the Hirzebruch-Jung continued fraction (see discussion preceeding \cite[Definition 2.3]{cyc}) of $n/a$ has exactly one term, i.e., $n=a\alpha_1$. But gcd$(n,a)=1$, hence this is equivalent to $a=1$.  
 \end{proof} 
 
 \begin{prop}\label{twoul} The cyclic quotient surface singularity corresponding to the group $\dfrac 1 n (1,a)$, where $a\ge 2$, has exactly two indecomposable Ulrich modules if and only if $n=ab-1$ for some $b\ge 2$.  
\end{prop} 

\begin{proof} First, assume that $n=ab-1$, where $b\ge 2$, so $gcd(n,a)=1$. Also, we know $1\le a\le n-1$. Since $a\ge 2$, so the Hirzebruch-Jung continued fraction (see discussion preceeding \cite[Definition 2.3]{cyc}) of $n/a$ is $\dfrac na=b-\dfrac 1a=[b, a]$. Thus $r=2$ (where $r$ denotes the number of irreducible exceptional curve, see \cite[Theorem 2.4]{cyc}), and then by \cite[Theorem 1.5]{cyc}, the number $N$ of indecomposable Ulrich modules  satisfy $2\le N\le 2^{2-1}=2$. Hence $N=2$.   

Conversely, let $N=2$ be the number of indecomposable Ulrich modules. Then $r\le N=2 \le 2^{r-1}$ by \cite[Theorem 1.5]{cyc}. So, $r=2$.
Hence by \cite[Theorem 2.4]{cyc} the Hirzebruch-Jung continued fraction of $n/a$ is of the form $\dfrac n a=\alpha -\dfrac 1 \beta$, where $\beta\ge 2$. Then $n=a\alpha -\dfrac {a}{\beta}$, so $\beta|a$. Moreover, $n\beta=a(\alpha\beta-1)$. Since $a,n$ are coprime, so $a|\beta$. Thus $\beta=a$.  Hence $\dfrac n a=\alpha -\dfrac 1 a$ gives $n=\alpha a-1$.  
\end{proof}  

\begin{prop}\label{cyclic} Let $R=S^G$ be a cyclic  quotient surface singularity, where $G=\dfrac 1 n (1,a)$. If $\Ul(R)$ has exactly two indecomposable objects, then $R$ is $\us$.   
\end{prop}  
 \begin{proof} If $\Ul(R)$ has exactly two indecomposable objects, then $a\geq 2$ by Proposition \ref{1ul}. Hence the indecomposable Ulrich modules are $M_{n-1}$ and $M_{n-a}$ (see \cite[Corollary 3.11]{cyc}), both of which have rank $1$. Applying Proposition \ref{Ggroup} to $\X=\Ul(R)$ (see \cite[Corollary 4.7]{ddd}), it is enough to prove that $2[M_{n-1}]\ne 2[M_{n-a}]$ in $G(R)$. If possible, let $2[M_{n-1}]=2[M_{n-a}]$ in $G(R)$. Then $2\text{cl}(M_{n-1})=2\text{cl}(M_{n-a})$ in $Cl(R)$. Since these modules are of rank $1$ and reflexive, so this implies $2[M_{n-1}]=2[M_{n-a}]$ in $Cl(R)$. Then Lemma \ref{quotcl} implies $2(n-1)\equiv 2(n-a) (\mod n)$. Hence $n|2a-2$. By Proposition \ref{twoul} we have, $n=ab-1$ for some $b\ge 2$, hence $n\ge 2a-1$. Thus $n|2a-2$ implies $2a-2=0$, so $a=1$, contradiction!  
 \end{proof}     
 
\begin{chunk}\label{chu} Finally, we quickly remark that if $M$ is a torsion-free module over a complete local domain $R$, and if $M$ has rank $r$, then $M$ has at most $r$-many indecomposable summands. Indeed, let $M\cong \oplus_{i=1}^{r+1} N_i$, where the first $r$-many modules $N_1,...,N_r$ are indecomposable (hence non-zero by definition). Then each $N_1,...,N_r$ has positive rank (as they are torsion-free). So by rank calculation, $N_{r+1}$ has rank zero. But $N_{r+1}$ is also torsion-free, hence $N_{r+1}=0$.
\end{chunk} 
 
 Now we can prove Theorem \ref{surface2}. For the following proof, we will use the notion of special MCM modules, for the definition and basic properties of which, see \cite[Section 2]{scm}.  
 \begin{thm}\label{surface} Let $(R,\m,k)$ be a singular complete local Cohen--Macaulay ring of minimal multiplicity, and have dimension $2$. Consider the following two statements: 
 
 \begin{enumerate}[\rm(1)]
  \item $R$ is a cyclic quotient singularity and $\Ul(R)$ has at most two indecomposable objects.
     \item $R$ is $\us$. 
 \end{enumerate}
 
 Then $(1)\implies (2)$. If moreover $k=\mathbb C$, then $(2)\implies (1)$.
 \end{thm}
 
 \begin{proof} $(1)\implies (2)$: If $\text{ind}\Ul(R)$ has exactly one object, then we are done by \ref{1split}. When $\Ul(R)$ has exactly two indecomposable objects, we are done by Proposition \ref{cyclic}.    
 
 $(2) \implies (1)$: Since $R$ is $\us$ and has minimal multiplicity, so $R$ is an isolated singularity by  Theorem \ref{minuss}. Hence $R$ is a normal as $\dim R=2$. Now since $R$ is local, so $R$ is a domain. Now $R$ has a fundamental module $E$ (see \cite[Remark 6.4]{umm}), which is Ulrich by \cite[Lemma 4.5(1)]{umm}. Since $\syz^2_R k \in \syz \cm^{\times}(R)$, so by  \cite[Lemma 6.5(2)]{umm}  we have a short exact sequence $0\to \syz_R \m \to E^{\oplus n}\to N\to 0$ for some $N\in \cm(R)$. Since $E$ is Ulrich, so $N$ is Ulrich. Since $R$ is $\us$, so the sequence splits. So, $\syz_R \m \in \add_R(E)$. So, $\add_R(\syz^2_R k)\subseteq \add_R(E)\subseteq \Ul(R)$. By Theorem \ref{minuss} we also have, $\Ul(R)=\add_R(\syz^2_R k)$. Hence we conclude $\Ul(R)=\add_R(E)$. Since $R$ is an integral domain, and $E$ has rank $2$ (which can be seen from \cite[Remark 6.4(1)]{umm}), so $E$ has at most two indecomposable summands (by \ref{chu}). Hence $\Ul(R)$ has at most two indecomposable objects. So, it only remains to prove that $R$ is a cyclic quotient singularity, which by \cite[Theorem 11.12]{yoshino} is same as showing $E$ is decomposable. 
 Since $\syz \cm^{\times}(R)\subseteq \Ul(R)$, so $\syz\cm^{\times}(R)$ has finite type. Since $\syz \cm(R)$ is closed under finite direct sums and direct summands \cite[Lemma 2.14]{scm}, so every module in $\syz \cm(R)$ is a direct sum of a free module and a module from $\syz \cm^{\times}(R)$. Hence $\syz \cm(R)$ has finite type, so \cite[Corollary 3.3]{ncr} implies $R$ has rational singularities. Now since $R$ is $2$-dimensional and normal, so a module is MCM if and only if it is reflexive. By \cite[Corollary 2.9]{scm}, $(-)^*$ gives a one-to-one correspondence between non-free modules in $\syz\cm(R)$ and non-free special MCM modules. Moreover, a reflexive module $X$ is indecomposable if and only if $X^*$ is indecomposable, hence \cite[Corollary 2.9]{scm} actually gives a one-to-one correspondence between non-free indecomposable modules in $\syz\cm(R)$ and non-free indecomposable special MCM modules. Now if $E$ is indecomposable, then $\syz \cm^{\times}(R)\subseteq \Ul(R)=\add_R(E)$ shows $\syz \cm^{\times}(R)$ has only one indecomposable object. So, $\syz \cm(R)$ has only one non-free indecomposable object. Hence there is only one non-free indecomposable special MCM module, so there is only one exceptional curve for Spec$(R)$ \cite[Theorem 1.2]{wunram}. So, the resolution of singularity graph of Spec$(R)$ is isomorphic to that of a Veronese ring, which is a cyclic quotient singularity. Hence \cite[Korollar 2.12]{bries} implies $R$ itself is a cyclic quotient singularity, contradicting our assumption that $E$ is indecomposable. Thus $E$ must be decomposable, and we are done. 
\end{proof}

\begin{cor} Let $R\cong \mathbb C[[x_0,x_1,\cdots,x_d]]/(f)$ be a singular hypersurface of minimal multiplicity. Then $R$ is $\us$ if and only if $R\cong \mathbb C[[x_0,x_1,\cdots,x_d]]/(g+x_2^2+\cdots+x_d^2)$, where $g=x_0^2+x_1^2$ or $x_0^2+x_1^3$. 
\end{cor}    

\begin{proof} We notice $\dim R=d$. First assume that $R$ is $\us$. By Proposition \ref{hyp}, we know $R$ has finite CM type. Hence by \cite[Theorem 9.8]{lw} we have, $R\cong \mathbb C[[x_0,x_1,\cdots,x_d]]/(g+x_2^2+\cdots+x_d^2)$, where $g$ defines an ADE singularity. First, we assume $d=2.$ In this case by Theorem \ref{surface} we have, $R$ is a cyclic quotient surface singularity with at most two indecomposable Ulrich modules. Since $R$ is a cyclic quotient singularity and hypersurface, so the defining polynomial must be of $A_{n-1}$-type by \cite[Theorem 6.12, Theorem 6.18]{lw}. Write $R=\mathbb C[[X,Y]]^G$, where $G=\dfrac 1n(1,a)$. If $R$ has exactly one indecomposable Ulrich module, then Proposition \ref{1ul} implies $a=1$. This means $R$ is the $n$-th Veronese of $\mathbb C[[X,Y]]$. But $R$ is also a hypersurface, hence $n=2$. Thus in this case, $R\cong \mathbb C[[X,Y]]^{(2)} \cong \mathbb C[[x_0,x_1,x_2]]/(x_0^2+x_1^2+x_2^2)$. If $R$ has exactly two indecomposable objects, then $a \ge 2$ by Proposition \ref{1ul}, and so $n=ab-1$ by Proposition \ref{twoul}. Since $R$ is also $A_{n-1}$-singularity, so $G=\dfrac 1n(1,n-1)$ by \cite[Theorem 6.12]{lw}. Thus $n-1=a$ divides $ab=n+1$. Since $n+1=2+(n-1)$, so $\dfrac{n+1}{n-1}=\dfrac{2}{n-1}+1$. This implies $\dfrac{2}{n-1}$ is an integer. Hence $a=n-1\le 2$. Thus $a=n-1=2$, so $n=3$. Hence $R$ is $A_2$, i.e., $R\cong \mathbb C[[x_0,x_1,x_2]]/(x_0^2+x_1^3+x_2^2)$ by \cite[Theorem 6.18]{lw}. This finishes the $d=2$ case. If $d=1$, then $R\cong \mathbb C[[x_0,x_1]]/(g)$ being $\us$ implies $R^{\#}\cong \mathbb C[[x_0,x_1,z]]/(g+z^2)$ is $\us$ by Proposition \ref{branch}. Hence we are done by the $d=2$ case.  Similarly, if $d\ge 3$, then $R$ is the repeated double branched cover of $\mathbb C[[x_0,x_1,x_2]]/(g+x_2^2)$, which has minimal multiplicity. So, we are again done by Proposition \ref{branch} and the $d=2$ case.  

Now conversely, assume $R$ is one of the forms as described in the statement. First, we assume $d=2$. Then by \cite[Theorem 6.18, Theorem 6.12]{lw} we get that, $R$ is the invariant ring under the action of $G=\dfrac 12 (1,1)$ (in which case $R$ has exactly one indecomposable Ulrich module by Proposition \ref{1ul}) or $\dfrac 13 (1,2)$ (in which case $R$ has exactly two indecomposable Ulrich modules by Proposition \ref{twoul}). So, $R$ is $\us$ by Theorem \ref{surface}. Now if $d=1$, then $R^{\#}$ is exactly the $d=2$ case, and if $d\ge 3$, then $R$ is the repeated double branched cover of the $d=2$ case. Hence we are done by Proposition \ref{branch} and the $d=2$ case.  
\end{proof}  

For our next application, we recall the following definition.

\begin{chunk}{\cite[Definition 4.1, Corollary 4.4]{pr}} Let $(R,\m)$ be a local Cohen--Macaulay ring. The test ideal associated to $\cm(R)$, denoted by $\tau_{\cm(R)}(R)$, is given by $\cap_{M\in \cm(R)} \tr_R(M)$. 
\end{chunk} 

\begin{chunk}\label{trace}
 By definition and \cite[Proposition 2.8(iii)]{lindo} we notice that, if $\tau_{\cm(R)}(R)=\m$, then $\tr_R(M)=\m$ for every $M\in \cm(R)$ without free summand. Conversely, if $R$ is singular, and $\tr_R(M)=\m$ for every $M\in \cm(R)$ without free summand, then $\m \subseteq \tau_{\cm(R)}(R)\neq R$ (if $\tau_{\cm(R)}(R)= R$, then $\tr_R(M)=R$ for every indecomposable $M\in \cm(R)$, and so every indecomposable MCM module would be free by \cite[Proposition 2.8(iii)]{lindo}, so $R$ would be regular). Hence $\m = \tau_{\cm(R)}(R)$. 
\end{chunk}

 \begin{lem}\label{trcl} Let $R=k[[x,y]]^G$ be a cyclic quotient surface singularity with maximal ideal $\m$, where $G=\dfrac 1 n(1,a)$ such that $a$ divides $n+1$. Then $\tr_R(M)=\m$ for every $M\in \cm(R)$ that has no free summand. In particular, it holds that $\tau_{\cm(R)}(R)=\m$.    
 \end{lem}  

 \begin{proof} By Proposition \ref{1ul} and Proposition \ref{twoul}, $\Ul(R)$ has at most two indecomposable objects. By Theorem \ref{surface}, then $R$ is $\us$. So, by Theorem \ref{minuss} we have $\m\subseteq \ann \Ext^1_R(M,\mod R)$. Since $R$ is a cyclic quotient surface singularity, hence every indecomposable MCM module has rank $1$ (\cite[Section 2, Preliminaries]{cyc}). Let $M\in\cm(R)$ be such that it has no free summand. Hence $M$ is a direct sum of regular ideals. Thus by Proposition \ref{traceann}(2) we get, $\m \subseteq \tr_R(M)$. Since $M$ has no free summand, so $\tr_R(M)=\m$ by \cite[Proposition 2.8(iii)]{lindo}. 
 \end{proof} 

 The following characterizes the cyclic  quotient surface singularities for which $\tau_{\cm(R)}(R)=\m$. 
 
\begin{prop}\label{trquot} Let $R=\mathbb C[[x,y]]^G$ be a cyclic quotient surface singularity with maximal ideal $\m$. Then the following are equivalent.
 
 \begin{enumerate}[\rm(1)]

\item $G=\dfrac 1 n(1,a)$, where $a$ divides $n+1$. 
 
     \item $\Ul(R)$ has at most two indecomposable objects.

     \item $\tr_R(M)=\m$ for every maximal Cohen--Macaulay $R$-module $M$ without free summand. 

     \item $\tau_{\cm(R)}(R)=\m$.
     
     \item $\tr_R(M)=\m$ for every Ulrich $R$-module $M$.
 \end{enumerate}
   
 \end{prop}
 
 \begin{proof}
 $(1)\iff (2)$: Follows from Proposition \ref{1ul} and \ref{twoul}.  
 
 $(1)\implies (3)$: Follows from Lemma \ref{trcl}.  
 
 $(3)\iff (4)$: Follows from \ref{trace}.

 $(3)\implies (5)$: Obvious, as $R$ is singular.  
 
 $(5)\implies (2)$: Let $M$ be an indecomposable Ulrich $R$-module. Then $\m=\tr_R(M)$. Since $R$ is a cyclic quotient surface singularity, hence $M$ has rank $1$ (\cite[Section 2, Preliminaries]{cyc}). So, $M$ is isomorphic to a regular ideal. Thus $\m=\tr_R(M)=\ann_R \Ext^1_R(M,\mod R)$ by \cite[Theorem 1.1]{sde}. Thus $\m \Ext^1_R(M,\mod R)=0$ for all $M\in\Ul(R)$. Hence $R$ is $\us$ by Theorem \ref{minuss}, so  $\Ul(R)$ has at most two indecomposable objects by Theorem \ref{surface}.   
\end{proof}

\section{dimension one rings}

Throughout this section, $(R,\m,k)$ will denote a local Cohen--Macaulay ring of dimension $1$. For an $\m$-primary (i.e., regular) ideal $I$ of $R$, let $B_R(I):=\cup_{n\ge 1}(I^n:I^n)$ denote the blow-up of $I$ (\cite[Definition 4.3 and Remark 4.4]{dms}). When the ring $R$ in question is clear, we drop the suffix $R$ and simply write $B(I)$.


First, we record some preliminary facts about dimension $1$ Cohen--Macaulay local rings.  

\begin{lem}\label{blow}
Let $(R,\m,k)$ be a local Cohen--Macaulay ring  of dimension $1$. Let $I \subseteq R$ be an $\m$-primary ideal of $R$. Then $B(I)$ is a finite birational extension of $R$, and $B(I)=(I^n:I^n)\cong I^n$ for all $n\gg 0$.  
Moreover, it also holds that $\tr_R(I^n)\cong (I^n)^*$ for all $n\gg 0$.    
\end{lem}  

\begin{proof}
That $B(I)$ is a finite birational extension of $R$ follows from \cite[Proposition 1.1(i)]{lip}. As mentioned in the proof of \cite[Corollary 1.4(i)]{lip}, $B(I)=B(I^s)$ for all $s>0$. From \cite[Definition 1.3, Corollary 1.4(i),(ii)]{lip} we then get $(I^n:I^n)=B(I^n)=B(I)$ for all $n\gg 0$. Since $I^n$ is stable for $n\gg 0$ (\cite[ Corollary 1.4(i),(ii)]{lip}), so by \cite[Proposition 1.1(ii), Definition 1.3]{lip} we get $I^n\cong B(I^n)=B(I)=(I^n:I^n)$. 

Hence by \ref{conductor} and \cite[Proposition 2.4(1)]{trace} we get, $\tr_R(I^n)=(R:B(I))\cong B(I)^*\cong (I^n)^*$ for all $n\gg 0$. 
\end{proof} 

In the following, for an $\m$-primary ideal $I$ of $R$, $\Ul_I(R)$ will denote the subcategory of $I$-Ulrich $R$-modules in the sense of \cite[Definition 4.1]{dms}, which is the collection of all maximal Cohen--Macaulay $R$-modules $M$ such that $e(I,M)=\ell(M/IM)$, i.e.,  $M\cong IM$ (see \cite[Theorem 4.6]{dms}).  

The next lemma says that the conductor of the integral closure is $I$-Ulrich.   

\begin{lem}\label{dim1cond} Let $R$ be a local Cohen--Macaulay ring of dimension $1$. Let $I$ be a regular ideal of $R$. Then $\c\in \Ul_I(R)$.
\end{lem} 

\begin{proof} By \ref{conductor} we have $\overline R \c\subseteq \c$. Since $B(I)$ is module finite over $R$, so $R\subseteq B(I)\subseteq \overline R$. Hence $B(I)\c \subseteq \overline R\c \subseteq \c$, so $\c \in \cm(B(I))=\Ul_I(R)$ (see \cite[Theorem 4.6]{dms}). 
\end{proof}

\begin{lem}\label{cond} Let $R$ be a local Cohen--Macaulay ring of dimension $1$, which is analytically unramified. Then $\mathfrak c\cong \overline R$. In particular, $\mathfrak c^*\cong \mathfrak c$.
\end{lem} 

\begin{proof} By \ref{conductor} we have $\overline R \c\subseteq \c$. So, $\mathfrak c$ is an ideal of $\overline R$. Now as $R$ is reduced, so $\overline R$ is a product of PID s (see \cite[Corollary 2.1.13]{sh}). Hence $\overline R$ is a PIR, so $\mathfrak c=a\overline R$ for some $a\in R$. Since $\overline R$ is module finite over $R$, so $\mathfrak c$ contains a non-zero-divisor. Hence $a$ is a non-zero-divisor, so $\mathfrak c \cong \overline R$. Now by Lemma \ref{normal} we have $\overline R\cong \mathfrak c^*$, so  $\mathfrak c\cong \mathfrak c^*$.
\end{proof}    

We need one more general lemma before proving our first main result of this section.

\begin{lem}\label{cmreg} Let $S$ be a Cohen--Macaulay ring. Then $S$ is regular if and only if every short exact sequence in $\cm(S)$ splits.  
\end{lem} 

\begin{proof} If $S$ is a regular ring (not necessarily local), then by Auslander--Buchsbaum formula every maximal Cohen--Macaulay $S$-module is locally free, and hence projective. So, every short exact sequence of modules in $\cm(S)$ splits.  Conversely, assume $S$ is a Cohen--Macaulay ring (not necessarily local), and every short exact sequence of maximal Cohen--Macaulay $S$-modules splits. To show $S$ is regular, we have to show that $S_{\p}$ is regular local ring for every prime ideal $\p$ of $S$. Hence we have to show that $\text{pd}_{S_{\p}} S_{\p}/\p S_{\p}<\infty$ for every prime ideal $\p$ of $S$. So, fix a prime ideal $\p$ of $S$. Let $r:=\text{Rfd}_S (S/\p):=\sup_{\q\in \spec(S)}\{\depth S_{\q}-\depth (S/\p)_{\q}\}$ (which is known to be finite by \cite[Definition 4.5]{dom}). Let $M,N$ be the $r$ and $(r+1)$-th syzygies, respectively, in some $S$-free resolution of $S/\p$. Then $M \in \cm(S)$ (see \cite[proof of Corollary 4.7]{dom}), and hence $N\in \cm(S)$. Moreover, we have a short exact sequence $0\to N \to F \to M \to 0$ for some free $S$-module $F$. Now by assumption this sequence splits, so $M$ is a projective $S$-module. Thus $\text{pd}_S S/\p <\infty$. Hence $\text{pd}_{S_{\p}} S_{\p}/\p S_{\p}<\infty$, and we are done.  
\end{proof}    

Now we can prove our first main result of this section. In the following, we will often use that if $R$ is $1$-dimensional and analytically unramified, then $R$ admits a canonical ideal (\cite[Proposition 2.7]{gotal}). In the following, $(-)^{\dagger}$ denotes the canonical dual. We also notice that if $\dim R=1$, then $\widehat R$ is an isolated singularity if and only if $\widehat R$ is reduced , i.e., $R$ is analytically unramified. 

\begin{thm}\label{dim1uss} Let $(R,\m,k)$ be a local Cohen--Macaulay ring  of dimension $1$. 
Let $I\subseteq R$ be an $\m$-primary ideal of $R$. Consider the following conditions:
\begin{enumerate}[\rm(1)] 
\item Every short exact sequence in $\Ul_I(R)$ splits.
\item $B(I)$ is a regular ring.
\item $B(I)=\overline{R}$.
\item $\overline{R}\cong I^n$ for all $n\gg 0$.
\item $\mathfrak c \cong I^n$ for all $n\gg 0$. 
\item $\tr_R(I^n)=\mathfrak{c}$ and $I^n$ is reflexive for all $n\gg 0$. 
\item $R$ is analytically unramified (i.e., $\widehat R$ is an isolated singularity) and $\Ul_I(R)=\add_R (\overline R)$. 
\item  $\Ul_I(R)=\add_R (\mathfrak c)$. 
\item  $\operatorname {ind}\Ul_I(R)=\{\mathfrak c\}$. 
\item  $| \operatorname {ind}\Ul_I(R)|=1$.  
\end{enumerate}

Then $(1)$ up to $(6)$ are all equivalent and $(1) \implies (7) \implies (8)$. If $R$ is Henselian and $(7)$ holds, then $M\cong M^{\dagger}$ for every $M\in \Ul_I(R)$. If $R$ is Henselian, integral domain and analytically unramified, then $(1)$ up to $(10)$ are all equivalent and these conditions imply $M\cong M^*$ for all $M\in \Ul_I(R)$.   
\end{thm}   

\begin{proof} (1) $\iff $ (2): We know that, $B(I)$ is a Cohen--Macaulay ring (not necessarily local) of dimension $1$, which is a finite birational extension of $R$ (by Lemma \ref{blow}). Since a module over a Cohen--Macaulay ring of dimension $1$ is torsion-free if and only if it is maximal Cohen--Macaulay, hence by Lemma \ref{algmor} and \cite[Theorem 4.6]{dms} we get $\Hom_R(M,N)=\Hom_{B(I)}(M,N)$ for every $M,N\in \cm(B(I))=\Ul_I(R)$. Hence every short exact sequence of modules in $\Ul_I(R)$ splits if and only if every short exact sequence in $\cm(B(I))$ splits. Now this is equivalent to $B(I)$ being a regular ring by Lemma \ref{cmreg}. 
\\
(2)$\implies$(3): Since $B(I)$ is a finite extension of $R$ (by Lemma \ref{blow}), so this is an integral extension of $R$. Hence $B(I)\subseteq Q(R)$ implies that $B(I)\subseteq \overline{R}$. Since $B(I)$ is a regular ring, so $B(I)$ is a normal ring. Hence \cite[Tag 034M]{st} implies that $B(I)$ is integrally closed in $Q(B(I))=Q(R)$. Let $x\in \overline{R}$. So, $x\in Q(R)=Q(B(I))$ and $x$ is integral over $R$. Hence $x$ is integral over $B(I)$ as well, so $x\in B(I)$. Thus $\overline{R}\subseteq B(I)$, so $\overline{R}=B(I)$. 
\\   
(3)$\implies$(2): Since $B(I)$ is a finite extension of $R$ , so (3) implies that $\overline{R}$ is a finite extension of $R$. Hence $R$ is reduced (see \cite[Theorem 4.6(ii)]{lw}). Since $R\subseteq B(I)\subseteq Q(R)$, so $B(I)$ is reduced as well. Since $Q(R)=Q(B(I))$ and $B(I)=\overline{R}$, so $B(I)$ is integrally closed in $Q(B(I))=Q(R)$. Then \cite[Tag 030C]{st} implies that $B(I)$ is a normal ring. So, $B(I)$ satisfies $(R_1)$. Now $B(I)$ is a finite extension of $R$, so $\dim(B(I))=\dim(R)=1$. Hence $B(I)$ is a regular ring.
\\
So far, we have proved $(1)\iff (2) \iff (3)$.
\\
(3)$\implies$(4): Since Lemma \ref{blow} implies that $B(I)\cong I^n$ for all $n\gg 0$, so $\overline{R}\cong I^n$ for all $n\gg 0$.
\\
(4)$\implies (5)$: The assumption of (4) implies that $\overline R$ is a finite $R$-module. Hence $R$ is analytically unramified  (see \cite[Theorem 4.6(ii)]{lw}), so $\overline R \cong \c$ by Lemma \ref{cond}. 
\\
(5) $\implies$ (6): Note that, since $I^n$ contains a regular element, so (5) implies that $\c$ contains a regular element. Hence $\overline{R}$ is a finite extension of $R$, so $R$ is analytically unramified  (see \cite[Theorem 4.6(ii)]{lw}). Then Lemma \ref{cond} and Lemma \ref{normal} implies that $\c$ is a reflexive $R$-module. Thus $I^n$ is reflexive for all $n\gg 0$. Moreover, since $\c$ is a trace ideal (by \ref{conductor}), so $\tr_R(I^n)=\c$ for all $n\gg 0$.   
\\
(6) $\implies$ (3): Note that, since $I^n\subseteq\tr_R(I^n)$ and $I^n$ contains a regular element, so (6) implies that $\c$ contains a regular element. Hence $\overline{R}$ is a finite extension of $R$, so $R$ is analytically unramified  (see \cite[Theorem 4.6(ii)]{lw}). Now $B(I)\cong I^n$ for all $n\gg 0$ by Lemma \ref{blow}, so $B(I)$ is a finite extension of $R$. The assumption of (6) implies that $B(I)$ is reflexive. We also have, $(R:\overline R)=\c=\tr_R(I^n)=\tr_R(B(I))=(R:B(I))$ (the last equality holds by \ref{conductor}). Since $\overline R$ is also reflexive (see Lemma \ref{normal}), so by \cite[Prposition 2.4(4)]{trace} we get that, $\overline R=R:(R:\overline R)=R:(R:B(I))=B(I)$.  

So, now we have proved $(1)\iff(2)\iff(3) $ and $(3) \implies (4)\implies (5)\implies (6)\implies (3)$. Thus (1) up to (6) are all equivalent.  

$(1)\implies (7)$: Note that, $\Ul_I(R)=\cm(B(I))$ by \cite[Theorem 4.6]{dms}. Hence (2) and (3) implies that, $\Ul_I(R)=\cm(B(I))=\cm(\overline R)=\add_R(\overline R)$ and $\overline R$ is a module finite extension over $R$ (as $B(I)$ is a finite extension of $R$ by Lemma \ref{blow}). Hence $R$ is analytically unramified by \cite[Theorem 4.6(ii)]{lw}. This proves $(1)(\iff (2) \iff (3)) \implies (7)$.   

$(7)\implies (8)$: This follows from Lemma \ref{cond}. 

Now suppose $R$ is Henselian and (7) holds. Let $M\in \Ul_I(R)=\add_R(\overline R)$ be an indecomposable $R$-module. Then $M$ is a direct summand of $(\overline R)^{\oplus n}$ for some $n\ge 1$. Since $R$ is Henselian, so Krull-Remak-Schmidt holds for $R$, i.e., indecomposable decomposition is unique. Hence by Lemma \ref{indec} we have, $M\cong \overline{R/\p}$ for some minimal prime $\p$ of $R$. Let $S=\overline{R/\p}$. Then by Lemma \ref{indec} we get that, $S$ is a  finitely generated R-module. Since $R$ is Henselian, so $R/\p$ is also Henselian (see \cite[Tag 04GH]{st}), and it is a local domain. Since $R$ is analytically unramified and $\p$ is a minimal prime of $R$, so $R/\p$ is analytically unramified (see \cite[Tag 032Y]{st}). Hence $\overline{R/\p}$ is a finite extension of $R/\p$ (see \cite[Tag 032Y]{st}). Since $R/p$ is Henselian, so every finite ring extension of $R/\p$ is a product of local rings, so $\overline{R/\p}$ is a product of local rings. But since $\overline{R/\p}$ is a domain, so $\overline{R/\p}$ is itself a local domain. Hence $S$ is a local domain. Moreover, $S$ is an integral domain integrally closed in its ring of fractions, so $S$ is normal by \cite[Tag 030C]{st}. Thus $S$ is a normal local domain of dimension $1$, hence $S$ is regular. Let $\omega_R$ be a canonical ideal of $R$. Then we have an isomorphism $\Hom_R(M,\omega_R) \cong \Hom_R(S, \omega_R)$. Since $S$ is a module finite $R$-algebra with $\dim(R)=\dim(S)=1$ and $S$ is regular local, so by \cite[Theorem 3.3.7]{bh} we have $\Hom_R(S,\omega_R)\cong \omega_S \cong S$. Thus $S \cong \Hom_R(S, \omega_R)\cong \Hom_R(M, \omega_R)$. Hence $M\cong S\cong M^{\dagger}$.    

Now moving forward, we assume $R$ is an integral domain, which is analytically unramified (i.e., $\overline R$ is module finite over $R$) and Henselian (so, Krull-Remak-Schmidt holds for $R$).  

$(8) \implies (9)$: Since $R$ is an integral domain and $\mathfrak c$ is an ideal of $R$, hence $\mathfrak c$ is an indecomposable $R$-module. Now let, $M \in \Ul_I(R)$ be indecomposable, so $M\in \add_R(\mathfrak c)$. Hence $M\oplus N\cong \mathfrak c^{\oplus n}$ , so by Krull-Remak-Schmidt we get $M\cong \mathfrak c$.   

$(9) \implies (10)$: Obvious. 

$(10) \implies (1)$: As in the above argument for $(8) \implies (9)$ we have, $\mathfrak c \in \Ul_I(R)$ is an indecomposable $R$-module. Hence $| \text{ind} \Ul_I(R)|=1$ implies $\text{ind} \Ul_I(R)=\{\mathfrak c\}$. Since $\Ul_I(R)$ is closed under direct summands, so every $M\in \Ul_I(R)$ is a direct sum of indecomposable $I$-Ulrich modules. Hence every non-zero module in $\Ul_I(R)$ is of the form $\mathfrak c^{\oplus n}$, and this has non-zero rank (since $\mathfrak c\ne 0$). Hence, if we have a short exact sequence of modules in $\Ul_I(R)$, say $0\to \mathfrak c^{\oplus a} \to \mathfrak c^{\oplus b}\to \mathfrak c^{\oplus f}\to 0$, then by rank consideration we get $b=a+f$. Thus $\mathfrak c^{\oplus b} \cong \mathfrak c^{\oplus a} \oplus \mathfrak c^{\oplus f}$, so the sequence splits. 

Finally, if (9) holds, then $M\cong M^*$ by Lemma \ref{cond}. 
\end{proof}     

The next result characterizes all numerical semigroup rings which are $\us$.

\begin{prop} Let $k$ be a field and $1<a_1<a_2<\cdots <a_n$ be integers with $\operatorname{gcd}(a_1,a_2,\cdots,a_n)=1$. Consider the numerical semigroup ring $R=k[[t^{a_1},\cdots, t^{a_n}]]$. Then $R$ is $\us$ if and only if $a_2=1+a_1$.   
\end{prop}  

\begin{proof} Since $R$ is a numerical semigroup ring, so $\overline R=k[[t]]$. By Theorem \ref{dim1uss}, $R$ is $\us$ if and only if $B(\m)=\overline R$, i.e., if and only if $B(\m)=k[[t]]$. By \cite[Remark 4.4]{dms}, if $x$ is a principal reduction of $\m$, then $B(\m)=R\left[\dfrac {\m}{x}\right]$. Now $t^{a_1}$ is a principal reduction of $\m$, hence $B(\m)=R\left[\dfrac{\m}{t^{a_1}}\right]=k[[t^{a_1},\cdots , t^{a_n},t^{a_2-a_1},\cdots,t^{a_n-a_1}]]=k[[t^{a_1},t^{a_2-a_1},\cdots,t^{a_n-a_1}]]$. Now note that, $a_1>1$, and $a_2-a_1<\cdots <a_n-a_1$.  So, $k[[t]]=B(\m)$ if and only if $t=t^{a_2-a_1}$ if and only if $a_2=1+a_1$.
\end{proof} 

Next, we prove that the $\us$ property for $1$-dimensional local Cohen--Macaulay rings ascend to completion. First, we need the following lemma.

\begin{lem}\label{addcomplete} Let $(R,\m)$ be a local Cohen--Macaulay ring of dimension $1$. If $X\in \Ul(\widehat R)$ such that $X\in \add_{\widehat R}(\widehat M)$ for some $M\in \cm(R)$, then $X\in \add_{\widehat R} (\widehat N)$ for some $N\in \Ul(R)$.  
\end{lem}

\begin{proof} Let $X\oplus Y\cong (\widehat M)^{\oplus l}$ for some $Y\in \mod(\widehat R)$ and for some integer $l\ge 1$. First, recall that for any $R$-modules $A,B$ and ideal $I$ of $R$, it holds that $I(A\oplus B)=IA \oplus IB,$ and also, $A\cong B \implies IA\cong IB$. So, $\widehat \m^n X \oplus \widehat \m^n Y\cong \widehat \m^n(\widehat M)^{\oplus l}\cong (\widehat \m^n \widehat M)^{\oplus l}\cong (\widehat{\m^n M})^{\oplus l}$ for every integer $n\ge 1$. Since $M\in \cm(R)$, so $\m^n M\in \Ul(R)$ for all $n\gg 0$, and moreover $\widehat \m^n X\cong X$ since $X\in \Ul(\widehat R)$. So, this completes the proof.     
\end{proof}  

For $1$-dimensional local Cohen--Macaulay rings, we can improve  Corollary \ref{hens} as follows.

\begin{cor}\label{complete} Let $(R,\m)$ be a local Cohen--Macaulay ring of dimension $1$. Then $R$ is $\us$ if and only if $\widehat R$ is $\us$. 
\end{cor}

\begin{proof} Notice that, for an MCM $R$-module $M$ we have,  $M\in \Ul(R)$ if and only if $\widehat M\in \Ul(\widehat R)$ by Remark \ref{faithflat}. Now we know that completion passes along short exact sequences and $\widehat M \cong \widehat N$ if and only if $M\cong N$ (see \cite[Corollary 1.15(ii)]{lw}). Hence if $\widehat R$ is $\us$, then so is $R$.  

Now suppose $R$ is $\us$. Then $\widehat R$ has isolated singularity by Theorem \ref{dim1uss}. Moreover, $0=\widehat{\m \Ext^1_R(M,N)}\cong \widehat \m \Ext^1_{\widehat R}(\widehat M, \widehat N)$ for all $M,N\in \Ul(R)$ by \cite[Proposition 5.2.2]{ddd}. Also, by \cite[Corollary 3.6]{stable} every $X\in \Ul(\widehat R)$ is the direct summand of $\widehat M$ for some $M\in \cm(R)$. Hence for every $X\in \Ul(\widehat R)$, there is $M\in \Ul(R)$ such that $X\in \add_{\widehat R}(\widehat M)$ (by Lemma \ref{addcomplete}). So, for every $X,Y\in \Ul(\widehat R)$ we have, $\Ext^1_{\widehat R}(X,Y)\in \add_{\widehat R} (\Ext^1_{\widehat R}(\widehat M, \widehat N))$ for some $M,N\in \Ul(R).$ Hence $\widehat \m \Ext^1_{\widehat R}(X,Y)=0$ for all $X,Y\in \Ul(\widehat R)$. Thus $\widehat R$ is $\us$ by \cite[Proposition 5.2.6]{ddd}.   
\end{proof}  

In the following, for a local Cohen--Macaulay ring $(R,\m, k)$ with $\m$ admitting a principal reduction (for example, when $k$ is infinite), let $rn(\m)$ denote the reduction number of $\m$, i.e., this is the smallest integer $n\ge 0$ such that $\m^{n+1}=x\m^n$ for some $x\in \m$. In particular, $\m^n \cong \m^{rn(\m)}$ for all $n\geq rn(\m)$.   

\begin{cor}\label{ulcmann1dim} Let $(R,\m, k)$ be a local Cohen--Macaulay ring of dimension $1$ and assume $k$ is infinite.  If $\m \Ext^1_R(\Ul(R),\Ul(R))=0$, then $\m^{rn(\m)}\Ext^1_R(\cm(R),\mod(R))=0$. 
\end{cor}  

\begin{proof} First, we deal with the case when $R$ is complete. By \cite[Proposition 5.2.6]{ddd}, $\m \Ext^1_R(\Ul(R),\Ul(R))=0$ implies $R$ is $\us$. Hence $R$ is analytically unramified and $B(\m)\cong \mathfrak c$ by Theorem \ref{dim1uss}.  By Lemma \ref{blow}, we get $\m^{rn(\m)}\cong \mathfrak c$. Hence $\m^{rn(\m)}\subseteq \tr_R (\m^{rn(\m)})= \tr_R(\mathfrak c)=\mathfrak c$. Since $R$ is complete and reduced, so by \cite[Proposition 3.1]{W} we get $\m^{rn(\m)}\Ext^1_R(\cm(R),\mod(R))\subseteq \mathfrak c \Ext^1_R(\cm(R),\mod(R))=0$.    

Now for the general case, $\m \Ext^1_R(\Ul(R),\Ul(R))=0$ again implies $R$ is $\us$ by \cite[Proposition 5.2.6]{ddd}. Then Corollary \ref{complete} implies $\widehat R$ is $\us$, and hence by \cite[Proposition 5.2.2]{ddd} we get, $\widehat {\m}\Ext^1_{\widehat R}(\Ul(\widehat R), \Ul(\widehat R))=0$. So, now the complete case that we have already proved implies $\widehat \m^{rn(\widehat \m)}\Ext^1_{\widehat R}(\cm(\widehat R),\mod(\widehat R))=0$. Since $rn(\widehat \m)\le rn(\m)$, so $\reallywidehat {\m^{rn(\m)}\Ext^1_R(M,N)}\cong \widehat \m^{rn(\m)} \Ext^1_{\widehat R}(\widehat M, \widehat N)=0$ for all $M\in \cm(R),N\in \mod(R)$. Hence $\m^{rn(\m)}\Ext^1_R(\cm(R),\mod(R))=0$.  
\end{proof} 

Finally, for one-dimensional local Cohen--Macaulay rings of minimal multiplicity, combining many of the previous results, we get the following characterization.  

\begin{cor}\label{dimsyzann} Let $(R,\m)$ be a singular local Cohen--Macaulay ring of dimension 1. Then the following are equivalent:
    
    \begin{enumerate}[\rm(1)]
    \item $R$ is analytically unramified and $\m\Ext^1_R(\Ul(R)\cap \syz \cm^{\times}(R), \syz \cm^{\times}(R))=0$.
        \item $R$ is analytically unramified and $\m \Ext^1_R(\c, \syz_R \c)=0$. 
        
        \item  $\m=\c$. 

        \item $(\m :\m)=\overline R$.  

        \item $\End_R(\m)\cong \overline R$.
        
        \item $\m \cong \c$.
        
        \item $R$ is $\us$ and has minimal multiplicity.
        
        \item  $R$ is analytically unramified and $\syz \cm^{\times}(R)=\add_R(\m)$.  
    \end{enumerate}   
    
    \end{cor}  
    
    \begin{proof}  
    By Lemma \ref{dim1cond} we have, $\c \in \Ul(R)$, and moreover, it is a trace ideal by  \ref{conductor}.     
    
    (1) $\implies$ (2): We have $\c=(R:\overline R)\cong \Hom_R(\overline R, R)$ (the last isomorphism follows from \cite[Proposition 2.4(1)]{trace}). So, $\c$ is the dual of a finitely generated $R$-module. Hence $\c\in \syz^2\mod R\subseteq \syz \cm(R)$.   Since $\c$ is an Ulrich module and $R$ is singular, so $\c$ has no free summand. Thus, $\c \in \Ul(R)\cap \syz \cm^{\times}(R)$. Also, $\syz_R \c\in \syz \cm^{\times}(R)$. Thus $\m \Ext^1_R(\c, \syz_R \c)=0$.  
    
    (2) $\implies $ (3): Since $\overline R$ is module finite over $R$, so $\c$ contains a non-zero-divisor. Hence by \cite[Theorem 1.1]{sde}, we have $\ann \Ext^1_R(\c,\syz_R \c) \subseteq \tr_R(\c)=\c$. Then by hypothesis of (2), we get $\m \subseteq \ann \Ext^1_R(\c,\syz_R \c)\subseteq\c$.  Hence $\c=\m$, or $\c=R$. But $R$ is singular, so $\c\ne R$. Hence $\c=\m$.  
    
    $(3) \implies (4)$: By hypothesis of (3), $\c$ contains a non-zero-divisor. Hence $R$ is analytically unramified by \cite[Theorem 4.6]{lw}.  Applying Lemma \ref{normal} we now get, $\overline R=(\c: \c)=(\m:\m)$. 

    $(4)\implies (5)$: This follows,  since $(\m :\m)\cong\End_R(\m)$ by \cite[Proposition 2.4(1)]{trace}.  

    $(5)\implies (6)$: By hypothesis of (5), $\overline R$ is module finite over $R$. Also, $\m^*\cong \End_R(\m)$ as $\tr_R(\m)=\m$. Hence $\m^*\cong \overline R$. Thus $\m \cong \m^{**}\cong (\overline R)^*\cong \c$, where the first isomorphism is by \cite[Theorem 4.1(2)]{restf} and the last isomorphism follows from \cite[Proposition 2.4(1)]{trace}. 
    
    $(6) \implies (7)$:  We have, $\m\subseteq \tr_R(\m)=\tr_R(\mathfrak c)=\mathfrak c\ne R$ (as $R$ is singular). Hence $\m=\c$ is Ulrich, so $R$ has minimal multiplicity by \cite[Proposition 2.5]{ul}. Hence $\mathfrak m^n \cong \mathfrak m \cong \mathfrak c$ for all $n\ge 1$. Thus by Theorem \ref{dim1uss} we get that, $R$ is $\us$. 
    
    $(7) \implies (8)$: By Theorem \ref{dim1uss} we  have, $R$ is analytically unramified and by Theorem \ref{minuss} we have, $\Ul(R)=\add_R(\m)$ . Since $\m \in \syz \cm^{\times}(R)$ and $\syz \cm^{\times}(R)$  is closed under direct summands of finite direct sums (see \cite[Lemma 5.5(2(b))]{umm}), so we have $\add_R(\m)\subseteq \syz \cm^{\times}(R)$. Now since $\syz \cm^{\times}(R) \subseteq \Ul(R)$, so we get the equality $\syz \cm^{\times}(R)=\add_R(\m)$. 
    
    $(8)\implies (1)$: Let $M\in \Ul(R)\cap \syz \cm^{\times}(R)\subseteq \syz \cm^{\times}(R)=\add_R(\m)$, so $M$ is a direct summand of direct sum of copies of $\m$. Hence for every $N\in\mod(R)$, we have $\Ext^1_R(M,N)$ is a direct summand of direct sum of copies of $\Ext^1_R(\m, N)\cong \Ext^2_R(k,N)$, which is annihilated by $\m$. Hence $\m\Ext^1_R(M,N)=0$ for all $M\in \Ul(R)\cap \syz \cm^{\times}(R)$ and $N\in\mod(R)$.
    \end{proof}

Next, we study the $\us$ property of fibre product of two rings. Let $R,S$ be two commutative noetherian rings, and let $I,J$ be proper ideals of $R,S$, respectively, such that $\frac{R}{I}\cong \frac{S}{J}(=:T)$. Then the fibre product $R\times_{T} S$ is defined by $R\times_{T} S:=\{(r,s)\in R\times S\text{ }:\text{ }r+I=s+J\}$, which is a subring of $R\times S$. In particular, if we take $R=S$ and $I=J$, then $R\times_{R/I} R:=\{(r_1,r_2)\in R\times R\text{ }:\text{ }r_1+I=r_2+I\}=\{(r_1,r_2)\in R\times R\text{ }:\text{ }r_1-r_2\in I\}=\{(r,r+i)\text{ }:\text{ }r\in R, i\in I\}$, which is just $R\bowtie I$ in the notation of \cite{anna}.
\\
First, we record the following two lemmas.

\begin{lem}\label{general fibre} 
Let $R,S$ be two commutative noetherian rings, and let $I,J$ be regular proper ideals of $R,S$, respectively, such that $\frac{R}{I}\cong \frac{S}{J}(=:T)$. Then the non-zero-divisors on $R\times_{T} S$ are of the form $(a,b)\in R\times_{T} S$ such that $a,b$ are non-zero-divisors on $R,S$ respectively. Moreover, $\dfrac{(r,s)}{(\sigma,\tau)}\mapsto\left(\dfrac{r}{\sigma},\dfrac{s}{\tau}\right)$ gives an isomorphism from $Q(R\times_{T} S)$ to $Q(R)\times Q(S)$.
\end{lem}

\begin{proof}
Let $f:\frac{R}{I}\to \frac{S}{J}$ be an isomorphism. Then we have $R\times_{T} S=\{(r,s)\text{ }:\text{ }f(\overline{r})=\overline{s}\text{ in }\frac{S}{J}\}$. 
\\
Since $I,J$ are regular ideals of $R,S$ respectively, so there exist $i\in I$ and $j\in J$ such that $i,j$ are non-zero-divisors on $R,S$ respectively. Let $(a,b)\in R\times_{T} S$ such that $a,b$ are non-zero-divisors on $R,S$ respectively. Now if $(r,s)\in R\times_{T} S$ with $(0,0)=(r,s)(a,b)=(ra,sb)$, then $ra=0,sb=0$, which implies $r=0$, $s=0$. So, $(r,s)=(0,0)$, which means $(a,b)$ is a non-zero-divisor on $R\times_{T} S$. Next let, $(a,b)$ be a non-zero-divisor on $R\times_{T} S$. First, we will show that $a$ is a non-zero-divisor on $R$. Let $c\in R$ be such that $ac=0$. Then $(ic,0)\in R\times_{T} S$ and $(a,b)(ic,0)=(0,0)$. Since $(a,b)$ is a non-zero-divisor on $R\times_{T} S$, hence $(ic,0)=(0,0)$, so $ic=0$. Since $i$ is a non-zero-divisor on $R$, so $c=0$. This implies that $a$ is a non-zero-divisor on $R$. Similarly, $b$ is a non-zero-divisor on $S$. Thus the non-zero-divisors on $R\times_{T} S$ are of the form $(a,b)\in R\times_{T} S$ such that $a,b$ are non-zero-divisors on $R,S$ respectively.
\\
Now consider the map $\phi:Q(R\times_{T} S)\to Q(R)\times Q(S)$ defined by $\phi\left(\dfrac{(r,s)}{(\sigma,\tau)}\right)=\left(\dfrac{r}{\sigma},\dfrac{s}{\tau}\right)$. Also, define $\psi:Q(R)\times Q(S)\to Q(R\times_{T} S)$ by $\psi\left(\dfrac{r}{\sigma},\dfrac{s}{\tau}\right)=\dfrac{(ri,sj)}{(\sigma i,\tau j)}$ $\Big($Here $\dfrac{(ri,sj)}{(\sigma i,\tau j)}\in Q(R\times_{T} S)$, since $f(\overline{ri})=f(\overline{0})=\overline{0}=\overline{sj}$, $f(\overline{\sigma i})=f(\overline{0})=\overline{0}=\overline{\tau j}$ in $\frac{S}{J}$ and $\sigma i,\tau j$ are non-zero-divisors on $R,S$ respectively as $\sigma, i$ are non-zero-divisors on $R$ and $\tau, j$ are non-zero-divisors on $S\Big)$. It is easy to check that, $\phi,\psi$ are well-defined ring homomorphisms. Now $(\phi\circ\psi)\left(\dfrac{r}{\sigma},\dfrac{s}{\tau}\right)=\phi\left(\dfrac{(ri,sj)}{(\sigma i,\tau j)}\right)=\left(\dfrac{ri}{\sigma i},\dfrac{sj}{\tau j}\right)=\left(\dfrac{r}{\sigma },\dfrac{s}{\tau}\right)$ and $(\psi\circ\phi)\left(\dfrac{(r,s)}{(\sigma,\tau)}\right)=\psi\left(\dfrac{r}{\sigma},\dfrac{s}{\tau}\right)=\dfrac{(ri,sj)}{(\sigma i,\tau j)}=\dfrac{(r,s)}{(\sigma,\tau)}\dfrac{(i,j)}{(i,j)}=\dfrac{(r,s)}{(\sigma,\tau)}$. Thus $\phi\circ\psi=\text{id}_{Q(R)\times Q(S)}$ and $\psi\circ\phi=\text{id}_{Q(R\times_{T} S)}$. Hence $\phi$ is an isomorphism, so $Q(R\times_{T} S)\cong Q(R)\times Q(S)$.
\end{proof}

\begin{lem}\label{fibre}
Let $R,S$ be two commutative noetherian rings and let $I,J$ be regular proper ideals of $R,S$ respectively such that $\frac{R}{I}\cong \frac{S}{J}(=:T)$. Then $I\times J$ is a regular ideal of $R\times_T S$. Moreover, we have isomorphism of rings $((I\times J)^n:_{Q(R\times_T S)}(I\times J)^n)\cong(I^n:_{Q(R)} I^n) \times(J^n:_{Q(S)}J^n)$ for all $n\ge 1$.
\end{lem}

\begin{proof}
Let $f:\frac{R}{I}\to \frac{S}{J}$ be an isomorphism. Then we have $R\times_{T} S=\{(r,s)\text{ }:\text{ }f(\overline{r})=\overline{s}\text{ in }\frac{S}{J}\}$. 
\\
Since for any $(i,j)\in I\times J$ we have $f(\overline{i})=f(\overline{0})=\overline{0}=\overline{j}$ in $\frac{S}{J}$, so $I\times J\subseteq R\times_T S$. Thus $I\times J$ is an ideal of $R\times_T S$, as $I\times J$ is an ideal of $R\times S$. Since $I,J$ are regular ideals of $R,S$ respectively, so there exist $i\in I$ and $j\in J$ such that $i,j$ are non-zero-divisors on $R,S$ respectively. Hence $(i,j)\in I\times J$ is a non-zero-divisor on $R\times_T S$, so $I\times J$ is a regular ideal of $R\times_T S$. Now from Lemma \ref{general fibre} we get that, the map $\phi:Q(R\times_{T} S)\to Q(R)\times Q(S)$ given by $\phi\left(\dfrac{(r,s)}{(\sigma,\tau)}\right)=\left(\dfrac{r}{\sigma},\dfrac{s}{\tau}\right)$ is a ring isomorphism. Consider the restriction of $\phi$ to $((I\times J)^n:_{Q(R\times_T S)}(I\times J)^n)$ and denote it by $\tilde{\phi}$. Since $\phi$ is a well-defined injective ring homomorphism, so $\tilde{\phi}$ is also a well-defined injective ring homomorphism.
\\
Next, we will show that $\tilde{\phi}(((I\times J)^n:_{Q(R\times_T S)}(I\times J)^n))=(I^n:_{Q(R)} I^n) \times(J^n:_{Q(S)}J^n)$. Let $\frac{(r,s)}{(\sigma,\tau)}\in ((I\times J)^n:_{Q(R\times_T S)}(I\times J)^n)$. Now to show $\frac{r}{\sigma}\in (I^n:_{Q(R)} I^n)$, we pick $a\in I^n$. Since $(I\times J)^n=I^n\times J^n$, so $(a,0)\in (I\times J)^n$. Then $(a,0)\cdot\left(\frac{(r,s)}{(\sigma,\tau)}\right)=\frac{(ar,0)}{(\sigma,\tau)}\in (I\times J)^n$. So, $\frac{(ar,0)}{(\sigma,\tau)}=\frac{(c,d)}{(1,1)}$ in $Q(R\times_T S)$ for some $(c,d)\in(I\times J)^n=I^n\times J^n$. Since $Q(R\times_T S)$ is obtained from $R\times_T S$ by inverting the non-zero-divisors of $R\times_T S$, so $(ar,0)=(\sigma c,\tau d)$. Hence $ar=\sigma c$, so $a\cdot \left(\frac{r}{\sigma}\right)=\frac{ar}{\sigma}=c\in I^n$ in $Q(R)$. Thus $\frac{r}{\sigma}\in (I^n:_{Q(R)} I^n)$. Similarly, $\frac{s}{\tau}\in (J^n:_{Q(S)}J^n)$. So, $\tilde{\phi}\left(\frac{(r,s)}{(\sigma,\tau)}\right)\in(I^n:_{Q(R)} I^n) \times(J^n:_{Q(S)}J^n)$, which implies $\tilde{\phi}(((I\times J)^n:_{Q(R\times_T S)}(I\times J)^n))\subseteq(I^n:_{Q(R)} I^n) \times(J^n:_{Q(S)}J^n)$.
\\
Now let, $\left(\frac{r}{\sigma}, \frac{s}{\tau}\right)\in (I^n:_{Q(R)} I^n) \times(J^n:_{Q(S)}J^n)$. Then from the proof of Lemma \ref{general fibre} we get that, $\frac{(ri,sj)}{(\sigma i,\tau j)}\in Q(R\times_T S)$ and  $\phi\left(\frac{(ri,sj)}{(\sigma i,\tau j)}\right)=\left(\frac{r}{\sigma}, \frac{s}{\tau}\right)$. Now to show that, actually $\frac{(ri,sj)}{(\sigma i,\tau j)}\in((I\times J)^n:_{Q(R\times_T S)}(I\times J)^n)$, we pick $(a,b)\in (I\times J)^n=I^n\times J^n$. Then $(a,b)\cdot \left(\frac{(ri,sj)}{(\sigma i,\tau j)}\right)=\frac{(ari,bsj)}{(\sigma i,\tau j)}$. Since $\frac{r}{\sigma}\in (I^n:_{Q(R)} I^n)$, so $a\cdot \left(\frac{r}{\sigma}\right)=\frac{ar}{\sigma}\in I^n$. Then $ar=a'\sigma$ for some $a'\in I^n$. Similarly, $bs=b'\tau$ for some $b'\in J^n$. This implies $\frac{(ari,bsj)}{(\sigma i,\tau j)}=\frac{(a'\sigma i,b'\tau j)}{(\sigma i,\tau j)}=(a',b')\cdot\frac{(\sigma i,\tau j)}{(\sigma i,\tau j)}=(a',b')\in I^n\times J^n=(I\times J)^n$. Thus $\frac{(ri,sj)}{(\sigma i,\tau j)}\in ((I\times J)^n:_{Q(R\times_T S)}(I\times J)^n)$. So, $\left(\frac{r}{\sigma}, \frac{s}{\tau}\right)\in \tilde{\phi}(((I\times J)^n:_{Q(R\times_T S)}(I\times J)^n))$. Thus $(I^n:_{Q(R)} I^n) \times(J^n:_{Q(S)}J^n)\subseteq\tilde{\phi}(((I\times J)^n:_{Q(R\times_T S)}(I\times J)^n))$. Hence $\tilde{\phi}(((I\times J)^n:_{Q(R\times_T S)}(I\times J)^n))=(I^n:_{Q(R)} I^n) \times(J^n:_{Q(S)}J^n)$, which implies the map $\tilde{\phi}:((I\times J)^n:_{Q(R\times_T S)}(I\times J)^n)\to (I^n:_{Q(R)} I^n) \times(J^n:_{Q(S)}J^n)$ is a ring isomorphism.
\end{proof}


\begin{prop}\label{fibbl}
Let $R,S$ be two local Cohen--Macaulay rings with 
$\dim(R)=\dim(S)=1$. Also, let $I,J$ be regular proper ideals of $R,S$ respectively such that $\frac{R}{I}\cong \frac{S}{J}(=:T)$. Then every short exact sequence in $\Ul_{I\times J}(R\times_T S)$ splits if and only if every short exact sequence in both $\Ul_{I}(R)$ and $\Ul_{J}(S)$ splits.
\end{prop}

\begin{proof}
Since $R,S$ are both local Cohen--Macaulay rings  and $\dim(R)=\dim(S)=1$, so $R\times_T S$ is also a local Cohen--Macaulay ring with $\dim(R\times_T S)=1$ (see \cite[Lemma 1.2, Lemma 1.5]{avra}). Now from Lemma \ref{blow} we have, $B(I)=(I^n:_{Q(R)} I^n)$, $B(J)=(J^n:_{Q(S)}J^n)$ and $B_{R\times_T S}(I\times J)=((I\times J)^n:_{Q(R\times_T S)}(I\times J)^n)$ for $n \gg 0$. Then from Theorem \ref{dim1uss} we get that, every short exact sequence in both $\Ul_{I}(R)$ and $\Ul_{J}(S)$ splits if and only if $B(I)$ and $B(J)$ are both regular rings if and only if $(I^n:_{Q(R)} I^n)$ and $(J^n:_{Q(S)}J^n)$ are both regular rings for $n \gg 0$ if and only if $(I^n:_{Q(R)} I^n)\times(J^n:_{Q(S)}J^n)$ is a regular ring for $n \gg 0$ if and only if $((I\times J)^n:_{Q(R\times_T S)}(I\times J)^n)$ is a regular ring for $n \gg 0$ by Lemma \ref{fibre} if and only if $B_{R\times_T S}(I\times J)$ is a regular ring if and only if every short exact sequence in $\Ul_{I\times J}(R\times_T S)$ splits.
\end{proof}

Now specialising the Proposition \ref{fibbl} by taking $I$, and $J$ to be the maximal ideals of $R$, and $S$ respectively, we obtain the following corollary.

\begin{cor}
Let $(R,\m_R,k)$ and $(S,\m_S, k)$ be two local Cohen--Macaulay rings with $\dim(R)=\dim(S)=1$. 
Then $R\times_k S$ is $\us$ if and only if both $R$ and $S$ are $\us$. 
\end{cor}

\section{Some applications for detecting finite projective or injective dimension}  
Let us recall that an $R$-module $M$ is said to satisfy $(S_n)$ if $\depth_{R_{\p}} M_{\p}\geq \inf\{n,\dim R_{\p}\}$ for all $\p \in \spec(R)$. 
The following lemma, which will be used subsequently, is perhaps well-known, but we  include a proof for the sake of completeness. 

\begin{lem}\label{codim1} Let $M,N$ be finitely generated modules over a Noetherian ring $R$. Then the following holds:         

\begin{enumerate}[\rm(1)]
    \item If $M$ satisfies $(S_1)$ and $f:M\to N$ is a map that is locally injective at minimal primes of $R$, then $f$ is injective.
    
    \item If $M,N$ satisfy $(S_2)$ and $f:M\to N$ is a map that is locally an isomorphism in co-dimension $1$ (i.e., at prime ideals of height at most $1$), then $f$ is an isomorphism. 
    
    \item If $M,N$ satisfy $(S_2)$, $N$ is locally free of rank $1$ in co-dimension $1$, and  moreover $M$ is locally reflexive in co-dimension $1$, then $\Hom_R(\Hom_R(M,N),N)\cong M$. So, in particular, if $R,N$ satisfy $(S_2)$, and $N$ is locally free of rank $1$ in co-dimension $1$, then $\Hom_R(N,N)\cong R$. 
     \item If $N$ satisfy $(S_2)$ and $M$ is locally free in co-dimension $1$, then $\Hom_R(M^*,\Hom_R(M,N))\cong \Hom_R(\Hom_R(M,M),N)$, and moreover, this is isomorphic to $N$ if $R,M$ also satisfy $(S_2)$. 
\end{enumerate}    
\end{lem}

\begin{proof} (1) Let $K:=\ker f$. If possible let, $K\ne 0$. Take $\p \in \Ass K \subseteq \Ass M$. Then $0=\depth M_{\p}\ge \inf \{1,\text{ht } \p\}$, so $\text{ht } \p=0$. Then by assumption, $f_{\p}$ is injective. Hence $K_{\p}=\ker f_{\p}=0$, contradicting $\p\in \Ass K\subseteq \supp K$. Thus $K=0$, i.e., $f$ is injective. 

(2)  By (1) we get that, $f$ is injective. Now let, $L=\text{coker} f$. If possible let, $L\neq 0$. Let $\p \in \Ass L$. Now the exact sequence $0\to M\to N\to L\to 0$ gives rise to the exact sequence $0\to M_{\p}\to N_{\p}\to L_{\p}\to 0$. Then the depth lemma gives $0=\depth L_{\p}\ge \inf \{\depth N_{\p}, \depth M_{\p}-1\}$. Now if ht $\p\ge 2$, then $\depth M_{\p}, \depth N_{\p}$ are at least $2$. Hence $0\ge \inf \{2,2-1\}=1$, impossible! Thus ht $\p\le 1$, and then by assumption $f_{\p}$ is an isomorphism, and so $L_{\p}=\text{coker} f_{\p}=0$, contradicting $\p\in \supp L$. Thus $\text{coker} f=0$, i.e., $f$ is also surjective.    

(3) First we show that, if $N$ satisfies $(S_2)$, then $\Hom_R(L,N)$ also satisfies $(S_2)$ for any finitely generated $R$-module $L$. Indeed, first assume that, ht $\p\ge 1$. Then $\depth N_{\p}\ge 1$, and hence $\depth (\Hom_R(L,N)_{\p}) \ge 1$ by \cite[Tag 0AV5]{st}. Similarly, if ht $\p \ge 2$, then $\depth N_{\p}\ge 2$, and so, \cite[Tag 0AV5]{st} implies $\depth (\Hom_R(L,N)_{\p}) \ge 2$. Thus $\Hom_R(L,N)$ satisfies $(S_2)$ for any finitely generated $R$-module $L$. So, in particular, $\Hom_R(\Hom_R(M,N),N)$ satisfies $(S_2)$. Now we have a natural map $M\to \Hom_R(\Hom_R(M,N),N)$ defined by $m\mapsto (f\mapsto f(m))$. This map is an isomorphism in co-dimension $1$, because in co-dimension $1$, $M$ is reflexive and $N$ is free of rank $1$. Hence by (2) we get that, this natural map is an isomorphism.  

(4) First note that, by Tensor-Hom adjunction we have, $\Hom_R(M^*\otimes_R M,N)\cong \Hom_R(M^*,\Hom_R(M,N))$. Also, there is a natural tensor evaluation map $M^*\otimes_R M\to \Hom_R(M,M)$ given by
$f\otimes m \mapsto (x\mapsto f(m)x)$ and this map is an isomorphism in co-dimension $1$, because $M$ is locally free in co-dimension $1$. Hence we have a natural map $\Hom_R(\Hom_R(M,M),N) \to \Hom_R(M^*\otimes_R M,N)$, which is an isomorphism in co-dimension $1$. Since $N$ satisfies $(S_2)$, so by the argument in (3) we know that, $\Hom_R(M^*\otimes_R M,N)$ and $\Hom_R(\Hom_R(M,M),N)$ also satisfy $(S_2)$. Hence by (2) we get that, $\Hom_R(M^*\otimes_R M,N) \cong \Hom_R(\Hom_R(M,M),N)$, and this proves the first isomorphism. Now if, moreover, $R,M$ satisfy $(S_2)$ and $M$ is locally free of rank $1$ in co-dimension $1$, then by (3) we know that $\Hom_R(M,M)\cong R$ . Hence we get, $\Hom_R(M^*\otimes_R M,N) \cong \Hom_R(\Hom_R(M,M),N)\cong N$.  
\end{proof} 

The first part of the following proposition refines part of \cite[Proposition 4.1]{umm} in a sense that, we do not assume $M$ is maximal Cohen--Macaulay, and moreover, we are also able to  say $\Hom_R(\syz_R M, N)$ is Ulrich.   

\begin{prop}\label{ulgen} Let $R$ be a local Cohen--Macaulay ring of dimension $d\ge 2$ and $M,N$ be $R$-modules.

\begin{enumerate}[\rm(1)] 

\item Assume $N$ is Ulrich. Let $\Ext^i_R(M,N)=0$ for all $1\le i\le d-1$. Then $\hom_R(M,N)$ and $\hom_R(\syz_R M,N)$ are Ulrich modules.  

\item Assume $M$ is Ulrich and $N$ is maximal Cohen--Macaulay, and $R$ admits a canonical module. Let $\Ext^i_R(M,N)=0$ for all $1\le i\le d-1$. Then $\hom_R(M,N)$ and $\hom_R( M, (\syz_R N^{\dagger })^{\dagger})$ are Ulrich modules. 
\end{enumerate}  
\end{prop}  

\begin{proof} (1) First, we prove the $d=2$ case. In this case, the assumption yields $\Ext^1_R(M,N)=0$. Consider the short exact sequence $0\to \syz_R M\to R^{\oplus n}\to M \to 0$ for some integer $n\ge 0$. Applying $\hom_R(-,N)$ we get a short exact sequence $0\to \hom_R(M,N) \to N^{\oplus n}\to \hom_R(\syz_R M,N)\to 0$. Since $N$ is Ulrich, so it is maximal Cohen--Macaulay. Hence $\depth N=2$. So, $\hom_R(M,N)$ and $\hom_R(\syz_R M,N)$ are maximal Cohen--Macaulay by \cite[Tag 0AV5]{st}. Now since $N^{\oplus n}$ is Ulrich, so $\hom_R(M,N)$ and $\hom_R(\syz_R M,N)$ are also Ulrich by \cite[Proposition 1.4]{ul}.    

Next, we assume that $d\ge 3$. Since $\Ext^1_R(M,N)=0$, so similar to above we get a short exact sequence $0\to \hom_R(M,N) \to N^{\oplus n}\to \hom_R(\syz_R M,N)\to 0$. Since $N$ is maximal Cohen--Macaulay and

$\Ext^i_R(\syz_R M,N)=\Ext^{i+1}_R(M,N)=0$ for all $1\le i\le d-2=(d-1)-1$, where $d-1>1$, so \cite[Lemma 2.3]{dao} implies that $\hom_R(\syz_R M,N)$ satisfies $(S_d)$, i.e., it is a maximal Cohen--Macaulay module. Now since $N$ is Ulrich, so from the short exact sequence $0\to \hom_R(M,N) \to N^{\oplus n}\to \hom_R(\syz_R M,N)\to 0$ we get that $\hom_R(M,N)$ and $\hom_R(\syz_R M,N)$ are Ulrich by \cite[Proposition 1.4]{ul}. 

(2) That $\Hom_R(M,N)$ is Ulrich follows from \cite[Proposition 4.1]{umm}. To prove $\hom_R( M, (\syz_R N^{\dagger })^{\dagger})$ is Ulrich, note that since $M,N$ are maximal Cohen--Macaulay, so $0=\Ext^{i}_R(M,N)\cong \Ext^{i}_R(N^{\dagger},M^{\dagger})$ for all $1\le i\le d-1$. Since $M^{\dagger}$ is Ulrich (\cite[Corollary 4.2]{umm}), so by part (1) we get that, $\Hom_R(\syz_R N^{\dagger}, M^{\dagger})$ is Ulrich. But $\Hom_R(\syz_R N^{\dagger}, M^{\dagger})\cong \hom_R( M, (\syz_R N^{\dagger })^{\dagger})$, hence the claim.  
\end{proof}  

Now we prove that, for rings which are $\us$, small number of vanishing of some particular Ext modules implies finite projective dimension. 

\begin{prop}\label{rigid}  Let $(R,\m)$ be a local Cohen--Macaulay ring of dimension $d\ge 2$, and assume $R$ is $\us$. Let $N$ be an Ulrich $R$-module, which is locally free of rank $1$ in co-dimension $1$. Let $M$ be an $R$-module such that $\Ext^i_R(M,N)=0$ for all $1\le i\le d-1$. If $\syz_R M$ is reflexive, then $\operatorname{pd} M\le 1$. If $M$ is reflexive, then $M$ is free. 
\end{prop} 

\begin{proof} First, consider the short exact sequence $0\to \syz_R M\to R^{\oplus n}\to M\to 0$ for some integer $n\ge 0$. Since $d-1\ge 1$, so $\Ext^1_R(M,N)=0$. Hence by taking $\Hom_R(-,N)$ we get a short exact sequence $0\to \hom_R(M,N) \to N^{\oplus n}\to \hom_R(\syz_R M,N)\to 0$, where all the modules are Ulrich by Proposition \ref{ulgen}(1). Since $R$ is $\us$, so this exact sequence splits. Hence $\Hom_R(M,N)$ and $\Hom_R(\syz_R M,N)$ are direct summand of $N^{\oplus n}$. Since $R$ is Cohen--Macaulay and $\syz_R M$ (resp. $M$) is reflexive, so $\syz_R M$ (resp. $M$) satisfies $(S_2)$. Hence by Lemma \ref{codim1}(3) we get, $\syz_R M\cong \Hom_R(\Hom_R (\syz_R M,N),N)$ (resp. $M\cong \Hom_R(\Hom_R (M,N),N)$), which is a direct summand of $\Hom_R(N,N)^{\oplus n}\cong R^{\oplus n}$. Thus $\syz_R M$ (resp. $M$) is free.
\end{proof}  

Since over generically Gorenstein Cohen--Macaulay rings second syzygy modules are always reflexive (see \cite[Theorem 2.3]{matsui}), so we get the following corollary.  

\begin{cor} Let $(R,\m)$ be a local Cohen--Macaulay ring of dimension $d\ge 2$, and let $R$ be $\us$. Assume $R$ is generically Gorenstein. Let $N$ be an Ulrich $R$-module, which is locally free of rank $1$ in co-dimension $1$. Let $M$ be an $R$-module. If $\Ext^{2\le i\le d}_R(M,N)=0$, then $\operatorname{pd} M\le 2$. If $\Ext^{i\le j\le d+i-2}_R(M,N)=0$ for some $i\ge 3$, then $\operatorname{pd } M\le i-1.$   
\end{cor}   

\begin{proof}  If $\Ext^{2\le i\le d}_R(M,N)=0$, then $\Ext^{1\le i\le d-1}_R(\syz_R M,N)=0$. Now $\syz_R(\syz_R M)$ is reflexive, hence applying Proposition \ref{rigid} we get that, $\operatorname{pd}(\syz_R M)\le 1$, so $\operatorname{pd} M\le 2$. Next, if $\Ext^{i\le j\le d+i-2}_R(M,N)=0$ for some $i\ge 3$, then $\Ext^{1\le j\le d-1}_R(\syz^{i-1}_R M,N)=0$. Now $\syz^{i-1}_R M$ is a second syzygy as $i\ge 3$, hence it is reflexive. Then Proposition \ref{rigid} gives $\syz^{i-1}_R M$ is free, so $\operatorname{pd } M\le i-1$.  
\end{proof}

\begin{prop} Let $(R,\m)$ be a local Cohen--Macaulay ring of dimension $d\ge 2$, admitting a canonical module $\omega$, and assume $R$ is $\us$.  
Let $M$ be an Ulrich $R$-module that is locally free of rank $1$ in co-dimension $1$, and let $N\in \cm(R)$. If $\Ext^i_R(M,N)=0$ for all $1\le i\le d-1$, then $N$ is isomorphic to a direct sum of copies of the canonical module.  
\end{prop}         

\begin{proof} First, consider the short exact sequence $0\to \syz_R N^{\dagger}\to R^{\oplus n}\to N^{\dagger}\to 0$ for some integer $n\ge 0$. Now by taking $\Hom_R(-,\omega)$ we get the short exact sequence $0\to N \to \omega^{\oplus n}\to (\syz_R N^{\dagger})^{\dagger}\to 0$. Since $d-1\ge 1$, so $\Ext^1_R(M,N)=0$. Hence by taking $\Hom_R(M,-)$ we get the short exact sequence
$0\to \Hom_R(M,N)\to (M^{\dagger})^{\oplus n}\to \Hom_R(M,(\syz_R N^{\dagger})^{\dagger})\to 0$, where all the modules are Ulrich by Proposition \ref{ulgen}(2). Since $R$ is $\us$, so the sequence $0\to \Hom_R(M,N)\to (M^{\dagger})^{\oplus n}\to \Hom_R(M,(\syz_R N^{\dagger})^{\dagger})\to 0$ splits. Hence $\Hom_R(M,N)$ is a direct summand of $\Hom_R(M,\omega)^{\oplus n}$, and hence $\Hom_R(M^*,\Hom_R(M,N))$ is a direct summand of $\Hom_R(M^*,\Hom_R(M,\omega))^{\oplus n}$. So, by Lemma \ref{codim1}(4) we get that, $N\cong \Hom_R(M^*,\Hom_R(M,N))$ is a direct summand of $\omega^{\oplus n}\cong \Hom_R(M^*,\Hom_R(M,\omega))^{\oplus n}$. Hence $N^{\dagger}$ is a direct summand of $(\omega^{\dagger})^{\oplus n}\cong R^{\oplus n}$, and thus $N^{\dagger}$ is free. Hence $N\cong N^{\dagger \dagger}$ is isomorphic to a direct sum of copies of $\omega$. 
\end{proof}

\end{document}